\documentclass[10pt,a4paper]{article}
\usepackage[a4paper]{geometry}
\usepackage{amssymb,latexsym,amsmath,amsfonts,amsthm}
\usepackage{graphicx,comment}
\usepackage{epsfig}
\usepackage{tikz}
\usepackage{subcaption}
\usepackage[T1]{fontenc}
\newtheorem{theorem}{Theorem}[section]
\newtheorem{lemma}[theorem]{Lemma}

\theoremstyle{definition}
\newtheorem{definition}[theorem]{Definition}

\theoremstyle{remark}

\newtheorem{remark}[theorem]{Remark}

\numberwithin{equation}{section}

\title{Asymptotics of greedy energy sequences on the unit circle and the sphere}

\date{\today}

\author{Abey L\'{o}pez-Garc\'{i}a\footnotemark[1] \quad Ryan E. McCleary\footnotemark[2]}

\begin{document}

\maketitle

\renewcommand{\thefootnote}{\fnsymbol{footnote}}
\footnotetext[1]{Department of Mathematics, University of Central Florida, 4393 Andromeda Loop North, Orlando, FL 32816, USA. email: abey.lopez-garcia\symbol{'100}ucf.edu.}
\footnotetext[2]{Department of Mathematics, University of Central Florida, 4393 Andromeda Loop North, Orlando, FL 32816, USA. email: remccleary\symbol{'100}knights.ucf.edu.}

\begin{abstract} 
For a parameter $\lambda>0$, we investigate greedy $\lambda$-energy sequences $(a_{n})_{n=0}^{\infty}$ on the unit sphere $S^{d}\subset\mathbb{R}^{d+1}$, $d\geq 1$, satisfying the defining property that each $a_{n}$, $n\geq 1$, is a point where the potential $\sum_{k=0}^{n-1}|x-a_{k}|^{\lambda}$ attains its maximum value on $S^{d}$. We show that these sequences satisfy the symmetry property $a_{2k+1}=-a_{2k}$ for every $k\geq 0$. The asymptotic distribution of the sequence undergoes a sharp transition at the value $\lambda=2$, from uniform distribution ($\lambda<2$) to concentration on two antipodal points ($\lambda>2$). We investigate first-order and second-order asymptotics of the $\lambda$-energy of the first $N$ points of the sequence, as well as the asymptotic behavior of the extremal values $\sum_{k=0}^{n-1}|a_{n}-a_{k}|^{\lambda}$. The second-order asymptotics is analyzed on the unit circle. It is shown that this asymptotic behavior differs significantly from that of $N$ equally spaced points on the unit circle, and a transition in the behavior takes place at $\lambda=1$.

\smallskip

\textbf{Keywords:} Greedy $\lambda$-energy sequence, Leja sequence, maximal distribution, potential, binary representation, Riemann zeta function.

\smallskip

\textbf{MSC 2020:} Primary 31C20, 31B15; Secondary 11M06.

\end{abstract}

\section{Introduction}

Let $S^{d}=\{x\in\mathbb{R}^{d+1}: |x|=1\}$ be the unit sphere of dimension $d\geq 1$. In this paper, $d$ will always refer to the dimension of the sphere we are considering, and so it is always a positive integer. Let $\lambda>0$, and let $\omega_{N}=(y_{k})_{k=0}^{N-1}$ be a configuration of $N\geq 2$ points on $S^{d}$, not necessarily distinct. We define 
\begin{equation}\label{def:energy}
H_{\lambda}(\omega_{N}):=\sum_{0\leq i\neq j\leq N-1}|y_{i}-y_{j}|^{\lambda}=2\sum_{0\leq i<j\leq N-1}|y_{i}-y_{j}|^{\lambda},
\end{equation} 
where $|\cdot|$ indicates the Euclidean norm.

The goal of this paper is the investigation of sequences $(a_{n})_{n=0}^{\infty}\subset S^{d}$ that satisfy the following condition:
\begin{equation}\label{recurcond}
\sum_{k=0}^{n-1}|a_{n}-a_{k}|^{\lambda}=\max_{x\in S^{d}}\sum_{k=0}^{n-1}|x-a_{k}|^{\lambda},\qquad\mbox{for all}\,\,n\geq 1.
\end{equation}
One can view \eqref{recurcond} as a recursive algorithm that generates the entire sequence $(a_{n})_{n=0}^{\infty}$, starting from an initial input $a_{0}$. Note that for a given $n\geq 2$, the choice of a point $a_{n}\in S^{d}$ that satisfies \eqref{recurcond} need not be unique. When studying these sequences in the special framework of the unit circle $S^{1}$, we will assume for convenience that the first point in the sequence is
\begin{equation}\label{initcond}
a_{0}=1.
\end{equation}
Similarly defined sequences, satisfying a condition as in \eqref{recurcond} but instead minimizing the Riesz potential $\sum_{k=0}^{n-1}|x-a_{k}|^{-s}$, $s>0$, or logarithmic potential $\sum_{k=0}^{n-1}\log\frac{1}{|x-a_{k}|}$, include the well-known Leja (also called Leja-G\'{o}rski) sequences. These sequences were first studied from the point of view of their energy asymptotics and distribution by Edrei \cite{Edrei}\footnotemark[4]\footnotetext[4]{In \cite{Edrei}, Leja sequences on a compact set $E\subset\mathbb{C}$ are introduced in page 78. Edrei then shows that the normalized Vandermonde determinant $|V(a_{0},\ldots,a_{n-1})|^{2/n^2}$ approaches the transfinite diameter of $E$.}, Leja \cite{Leja}, G\'{o}rski \cite{Gorski}, and Siciak \cite{Siciak}. Recently, potential-theoretic properties of these extremal sequences were investigated in  \cite{CorDra, Gotz, HRSV, Lop, LopSaff, LopWag, Prit1, Prit2, ST}. In \cite{LopSaff}, more general \emph{greedy energy sequences} associated with general kernels were defined and studied. By analogy, we will also refer to the sequences investigated in the present work as \emph{greedy $\lambda$-energy sequences}. 

For a greedy $\lambda$-energy sequence $(a_{n})_{n=0}^{\infty}\subset S^{d}$ satisfying \eqref{recurcond}, we will use the notation
\begin{equation}\label{def:alphaNlambda}
\alpha_{N,\lambda}:=(a_{0},\ldots,a_{N-1}), \qquad N\geq 1.
\end{equation}   
We call $\alpha_{N,\lambda}$ the $N$-th section of the sequence. If the value of $\lambda$ is clear from the context, we may simply write $\alpha_{N}$, or $H(\alpha_{N})$ in reference to the energy \eqref{def:energy} of $\alpha_{N}$. 

We will denote by $\sigma_{N,\lambda}$ the normalized counting measure
\begin{equation}\label{def:countmeas}
\sigma_{N,\lambda}:=\frac{1}{N}\sum_{k=0}^{N-1}\delta_{a_{k}}
\end{equation}
associated with the $N$th section $\alpha_{N,\lambda}$, where $\delta_{a}$ is the Dirac unit measure at $a$. If $\mu$ and $(\mu_{n})_{n\in\mathbb{N}}$ are probability measures on $S^{d}$, then
\[
\mu_{n}\stackrel{\ast}{\longrightarrow}\mu
\]
means that for every continuous function $f:S^{d}\longrightarrow\mathbb{R}$, we have
\[
\lim_{n\rightarrow\infty}\int f\,d\mu_{n}=\int f\,d\mu.
\]

In this paper, $\mathcal{M}_{d}$ denotes the space of all probability distributions on $S^{d}$, and $\sigma_{d}\in\mathcal{M}_{d}$ denotes the normalized uniform measure on $S^{d}$. Given $\mu\in\mathcal{M}_{d}$, we define its $\lambda$-energy 
\begin{equation}\label{eq:defcontenerg}
I_{\lambda}(\mu):=\iint_{S^{d}\times S^{d}}|x-y|^{\lambda}\,d\mu(x)\,d\mu(y), \qquad \lambda>0.
\end{equation}
For a given $\lambda>0$, a distribution $\sigma\in\mathcal{M}_{d}$ is called \emph{maximal} if  
\[
I_{\lambda}(\sigma)=\sup_{\mu\in\mathcal{M}_{d}}I_{\lambda}(\mu).
\]
The existence of maximal distributions follows from a standard convergence argument in potential theory. Maximal distributions were described by Bj\"{o}rck in \cite{Bjorck}. We cite his result. 

\begin{theorem}[Bj\"{o}rck \cite{Bjorck}]\label{theo:Bjorck}
For $0<\lambda<2$, the measure $\sigma_{d}$ is the unique maximal distribution in $\mathcal{M}_{d}$. For $\lambda=2$, a distribution $\sigma\in\mathcal{M}_{d}$ is maximal if and only if its center of mass is at the origin, i.e., it satisfies \eqref{eq:centermass}. For $\lambda>2$, a distribution $\sigma\in\mathcal{M}_{d}$ is maximal if and only if it is of the form $\sigma=\frac{1}{2}(\delta_{a}+\delta_{-a})$ for some $a\in S^{d}$.
\end{theorem}

We remark that for $0<\lambda<2$, we have
\[
I_{\lambda}(\sigma_{d})=\frac{\Gamma(\frac{d+1}{2})\Gamma(d+\lambda)}{\Gamma(\frac{d+\lambda+1}{2})\Gamma(d+\frac{\lambda}{2})}=\frac{2^{d+\lambda-1}}{\sqrt{\pi}}\frac{\Gamma(\frac{d+1}{2})\Gamma(\frac{d+\lambda}{2})}{\Gamma(d+\frac{\lambda}{2})},
\]
see e.g. \cite[Proposition 4.6.4]{BorHarSaff}, and for $\lambda\geq 2$, and any corresponding maximal distribution $\sigma$,
\[
I_{\lambda}(\sigma)=2^{\lambda-1},
\]
see \cite[Theorems 6 and 7]{Bjorck}.

In this paper, $U_{n}(x)=U_{n}^{\lambda}(x)$ will denote the discrete potential
\begin{equation}\label{def:discrpot}
U_{n}(x):=\sum_{k=0}^{n-1}|x-a_{k}|^{\lambda},\qquad n\geq 1.
\end{equation}
Since the value of $\lambda$ will be clear from the context, we will omit the superscript $\lambda$ when indicating $U_{n}$. 

Our main goal in this work is the asymptotic analysis of the sequences $H_{\lambda}(\alpha_{N,\lambda})$, $U_{n}(a_{n})$, and $\sigma_{N,\lambda}$. We first show that for any values of $d\geq 1$ and $\lambda>0$, greedy $\lambda$-energy sequences $(a_{n})_{n=0}^{\infty}\subset S^{d}$ satisfy the symmetry property 
\begin{equation}\label{int:symprop}
a_{2k+1}=-a_{2k},\qquad \mbox{for all}\quad k\geq 0,
\end{equation}
see Theorem~\ref{theo:symmprop}. This property allows us to compute half the points of a greedy sequence automatically and concentrate the computational effort on obtaining the points $a_{2k}$ with even index. In the range $0<\lambda<2$, greedy sequences are uniformly distributed on $S^{d}$, see \eqref{eq:asympdistr}, whereas in the case $\lambda>2$, we have $\{a_{2k}, a_{2k+1}\}=\{a_{0}, a_{1}\}$ for all $k\geq 0$, so the sequence is restricted to its first two points, see Theorem~\ref{theo:caselamg2}. In the intermediate case $\lambda=2$, a sequence $(a_{n})$ is a greedy $2$-energy sequence if and only if it satisfies \eqref{int:symprop}, and so in this case it is possible that $(\sigma_{N,2})$ diverges, see Theorem~\ref{theo:caseletwo} and Remark~\ref{rmk:distrletwo}.

In the range $0<\lambda<2$, the first-order asymptotic behavior of the sequences $(H_{\lambda}(\alpha_{N,\lambda}))$ and $(U_{n}(a_{n}))$ is described by formulas \eqref{eq:leadenergasymp} and \eqref{asympUnan}. These formulas immediately suggest the study of the sequences $(H_{\lambda}(\alpha_{N,\lambda})-N^{2} I_{\lambda}(\sigma_{d}))$ and $(U_{n}(a_{n})-n I_{\lambda}(\sigma_{d}))$, or second-order asymptotics, which we develop in detail in dimension $d=1$. The asymptotic analysis of the sequence $(H_{\lambda}(\alpha_{N,\lambda})-N^{2} I_{\lambda}(\sigma_{1}))$ needs to be divided in three regimes; $0<\lambda<1$, $\lambda=1$, $1<\lambda<2$, as we have different orders of growth $H_{\lambda}(\alpha_{N,\lambda})-N^{2} I_{\lambda}(\sigma_{1})=O(\kappa_{\lambda}(N))$, $N\rightarrow\infty$, where 
\[
\kappa_{\lambda}(N)=\begin{cases}
N^{1-\lambda} & 0<\lambda<1,\\
\log N & \lambda=1,\\
1 & 1<\lambda<2.
\end{cases}
\]
It turns out that the bounded sequences $((H_{\lambda}(\alpha_{N,\lambda})-N^{2} I_{\lambda}(\sigma_{1}))/\kappa_{\lambda}(N))$ are also divergent, see Theorems~\ref{theo:casello}, \ref{theo:caselgo}, and \ref{theo:lambdaeq1}. 

Let $\mathcal{L}_{\lambda}(N)$ denote the $\lambda$-energy of the configuration formed by the $N$th roots of unity on the unit circle.\footnotemark[4]\footnotetext[4]{In the range $0<\lambda<2$, $\mathcal{L}_{\lambda}(N)$ is the largest value for the $\lambda$-energy of an $N$-point configuration on $S^{1}$.} If we compare, in the range $0<\lambda<2$ and on the unit circle, the asymptotics of $H_{\lambda}(\alpha_{N,\lambda})$ and $\mathcal{L}_{\lambda}(N)$, we see that they coincide in first-order asymptotics, but they differ significantly at the level of second-order asymptotics. Indeed, we know by \cite[Theorem 1.1]{BrauHardSaff} that $\mathcal{L}_{\lambda}(N)-I_{\lambda}(\sigma_{1})N^{2}=O(N^{1-\lambda})$ for all $\lambda\in(0,2)$, and the sequence $(\mathcal{L}_{\lambda}(N)-I_{\lambda}(\sigma_{1})N^{2})/N^{1-\lambda}$ converges. 

In the range $0<\lambda<2$, we show that the sequence $(U_{n}(a_{n})-n I_{\lambda}(\sigma_{1}))$ is bounded and divergent, and we have $0<U_{n}(a_{n})-n I_{\lambda}(\sigma_{1})<I_{\lambda}(\sigma_{1})$ for all $n\geq 1$, where the bounds are sharp, see Theorem~\ref{theo:normpot}. 

We illustrate graphically the sequences $(U_{n}(a_{n})-n I_{\lambda}(\sigma_{1}))$ (see Figs.~\ref{plotsecord5}, \ref{plotsecord6}) and $((H_{\lambda}(\alpha_{N,\lambda})-N^{2} I_{\lambda}(\sigma_{1}))/\kappa_{\lambda}(N))$ (see Figs.~\ref{plotsecorder2}, \ref{plotsecord4}, \ref{plotsecorder3}). These plots show a rather regular and rich behavior of these sequences, which was not anticipated by the authors.

The organization of this paper is simple. In Section~\ref{sec:symmprop} we prove the symmetry property. In the rest of the sections, we organize the results obtained in the different ranges for $\lambda$, namely $0<\lambda<2$, $\lambda=2$, and $\lambda>2$. 

\section{Symmetry property}\label{sec:symmprop}

\begin{theorem}\label{theo:symmprop}
Let $\lambda>0$ and $d\geq 1$ be arbitrary, and let $(a_n)_{n=0}^{\infty} \subset S^d$ be a greedy $\lambda$-energy sequence. For any odd index $n=2k+1$, $k\geq 0$, the function $U_{n}(x)=\sum_{j=0}^{n-1}|x-a_{j}|^{\lambda}$ attains its maximum value at a unique point on $S^{d}$, which is $-a_{n-1}=-a_{2k}$, i.e., we have   
\begin{equation}\label{eq:symmrel}
a_{2k+1}=-a_{2k},\qquad \mbox{for all}\quad k\geq 0.
\end{equation}
For any even index $n=2k\geq 2$, we have
\begin{equation}\label{eq:symmpoteven}
U_{2k}(x)=U_{2k}(-x),\qquad \mbox{for all}\,\,x\in S^{d}.
\end{equation}
\end{theorem}
\begin{proof}
First, it is obvious that $a_1 = -a_0$. Now assume that $a_{2k+1} = -a_{2k}$ for all $0 \leq k \leq l$. Let
\[
U(x) := \sum_{k=0}^{l} \left( |x - a_{2k}|^{\lambda} + |x - a_{2k+1}|^{\lambda} \right),\qquad x\in S^{d}.
\]
By assumption, this can be rewritten as
\[
U(x) = \sum_{k=0}^{l} \left( |x - a_{2k}|^{\lambda} + |x + a_{2k}|^{\lambda} \right),
\]
hence $U(x)=U(-x)$ for all $x\in S^{d}$. By definition, $a_{2l+2}$ must maximize $U$. We will show now that the choice of $a_{2l+3}$ is unique and that $a_{2l+3} = -a_{2l+2}$.

By definition, $a_{2l+3}$ must maximize the function
\[
U(x) + |x - a_{2l+2}|^{\lambda}\qquad x\in S^{d}.
\]
Since $U(x)$ has a maximum at $-a_{2l+2}$ and the function $x \mapsto |x - a_{2l+2}|^{\lambda}$ has a unique maximum at 
$-a_{2l+2}$, necessarily $a_{2l+3}=-a_{2l+2}$. We have proved \eqref{eq:symmrel} by induction, and \eqref{eq:symmpoteven} follows.
\end{proof}

\section{The case $0<\lambda<2$}\label{sec:firstorder}
In this section we present results valid in the range $0<\lambda<2$.

\subsection{First-order asymptotics and uniform distribution}

We describe in this subsection the first-order asymptotic behavior of the sequences $\left(H_{\lambda}(\alpha_{N,\lambda})\right)_{N=2}^{\infty}$ and $(U_{n}(a_{n}))_{n=1}^{\infty}$. It is also shown that in the range $0<\lambda<2$, greedy $\lambda$-energy sequences are uniformly distributed on $S^{d}$.

\begin{theorem}
Assume $0<\lambda<2$, let $d\geq 1$ be arbitrary, and let $(a_{n})_{n=0}^{\infty}\subset S^{d}$ be a greedy $\lambda$-energy sequence. The associated sequence of configurations \eqref{def:alphaNlambda} satisfies 
\begin{equation}\label{eq:boundenergy}
N(N-1) I_{\lambda}(\sigma_{d})< H_{\lambda}(\alpha_{N,\lambda})< N^{2} I_{\lambda}(\sigma_{d}),\qquad N\geq 2.
\end{equation}
In particular,
\begin{equation}\label{eq:leadenergasymp}
\lim_{N \to \infty} \frac{H_{\lambda}(\alpha_{N,\lambda})}{N^2}=I_{\lambda}(\sigma_{d}).
\end{equation}
For the sequence \eqref{def:countmeas} we have
\begin{equation}\label{eq:asympdistr}
\sigma_{N,\lambda}\stackrel{\ast}{\longrightarrow}\sigma_{d}.
\end{equation}
\end{theorem}
\begin{proof}
We follow the strategy used in the proof of Theorem 1.1 in \cite[Ch. V]{ST}. Let $U_{n}(x)$ be the function defined in \eqref{def:discrpot}. In virtue of \eqref{recurcond},
\begin{equation}\label{eq:ineqpot}
U_{n}(a_{n})\geq U_{n}(x),\qquad \mbox{for all}\,\,x\in S^{d},\quad n\geq 1.
\end{equation}
This implies
\begin{equation}\label{eq:ineqHU}
H_{\lambda}(\alpha_{N,\lambda})=2\sum_{0\leq k<i\leq N-1}|a_{i}-a_{k}|^{\lambda}=2\sum_{i=1}^{N-1}U_{i}(a_{i})\geq 2\sum_{i=1}^{N-1}U_{i}(x),\qquad x\in S^{d}.
\end{equation}
In fact, since $U_{1}(x)<U_{1}(a_{1})$ for all $x\in S^{d}\setminus\{a_{1}\}$, we have
\[
\sum_{i=1}^{N-1}U_{i}(a_{i})>\sum_{i=1}^{N-1}U_{i}(x),\qquad x\in S^{d}\setminus\{a_{1}\}.
\]
Integrating both sides of \eqref{eq:ineqHU} with respect to $\sigma_{d}$, we get
\begin{align}
H_{\lambda}(\alpha_{N,\lambda}) & > 2\sum_{i=1}^{N-1}\int U_{i}(x)\,d\sigma_{d}(x)\notag\\
& = 2\sum_{i=1}^{N-1}\sum_{k=0}^{i-1}\int |x-a_{k}|^{\lambda}\,d\sigma_{d}(x)\notag\\
& =2\sum_{i=1}^{N-1}\sum_{k=0}^{i-1}U^{\sigma_{d}}(a_{k})\label{eq:ineq1}
\end{align}
where
\[
U^{\sigma_{d}}(y):=\int |x-y|^{\lambda}\,d\sigma_{d}(x)
\]
is the potential of $\sigma_{d}$. By symmetry, 
\begin{equation}\label{eq:potconst}
U^{\sigma_{d}}(y)=I_{\lambda}(\sigma_{d})\quad\mbox{for all}\,\,y\in S^{d}. 
\end{equation}
This and \eqref{eq:ineq1} imply
\[
H_{\lambda}(\alpha_{N,\lambda})> N (N-1) I_{\lambda}(\sigma_{d}).
\]

Let $\omega_{N,\lambda}^{*}$ be an $N$-point configuration on $S^{d}$ (with possible point repetitions) that maximizes the energy \eqref{def:energy} among all such configurations, i.e.,
\[
H_{\lambda}(\omega_{N,\lambda}^{*})=\max\{H_{\lambda}(\omega_{N}): \omega_{N}\subset S^{d},\,\,\mathrm{card}(\omega_{N})=N\,\,\textrm{counting mult.}\}.
\]
Let
\[
\tau_{N}:=\frac{1}{N}\sum_{x\in\omega_{N,\lambda}^{*}}\delta_{x}.
\]
Applying Theorem~\ref{theo:Bjorck}, we have
\begin{equation}\label{ineqalphaomega}
H_{\lambda}(\alpha_{N,\lambda})
\leq H_{\lambda}(\omega_{N,\lambda}^{*})=N^{2} I_{\lambda}(\tau_{N})< N^{2} I_{\lambda}(\sigma_{d}).
\end{equation}
This finishes the proof of \eqref{eq:boundenergy} and \eqref{eq:leadenergasymp}. 

Let $\sigma$ be a limit point of the sequence $(\sigma_{N,\lambda})_{N=1}^{\infty}$ in the weak-star topology. Let $(\sigma_{N,\lambda})_{N\in\mathcal{N}}$ be a subsequence that converges to $\sigma$. Then
\[
\lim_{N\in\mathcal{N}}\iint|x-y|^{\lambda}\,d\sigma_{N,\lambda}(x)\, d\sigma_{N,\lambda}(y)=\iint |x-y|^{\lambda}\,d\sigma(x)\,d\sigma(y).
\]
This is equivalent to $\lim_{N\in\mathcal{N}} N^{-2} H_{\lambda}(\alpha_{N,\lambda})=I_{\lambda}(\sigma)$, hence $I_{\lambda}(\sigma)=I_{\lambda}(\sigma_{d})$. By the uniqueness of the maximal distribution, we get $\sigma=\sigma_{d}$. Since any limit point of the sequence $(\sigma_{N,\lambda})_{N=1}^{\infty}$ is $\sigma_{d}$, we obtain \eqref{eq:asympdistr}.
\end{proof}

\begin{remark}
In virtue of \eqref{eq:boundenergy}, the sequence $((H_{\lambda}(\alpha_{N,\lambda})-N^2\,I_{\lambda}(\sigma_{d}))/N)_{N=2}^{\infty}$ is bounded and we have
\[
-I_{\lambda}(\sigma_{d})< \frac{H_{\lambda}(\alpha_{N,\lambda})-N^{2}\,I_{\lambda}(\sigma_{d})}{N}< 0\qquad\mbox{for all}\,\,N\geq 2.
\] 
In the case of the unit circle $d=1$, our later results show that
\[
\lim_{N\rightarrow\infty}\frac{H_{\lambda}(\alpha_{N,\lambda})-N^{2}\,I_{\lambda}(\sigma_{1})}{N}=0.
\]
\end{remark}

Our next result concerns the sequence $(U_{n}(a_{n}))_{n=1}^{\infty}$.

\begin{theorem}\label{theo:asymppot}
Assume $0<\lambda<2$, let $d\geq 1$ be arbitrary, and let $(a_{n})_{n=0}^{\infty}\subset S^{d}$ be a greedy $\lambda$-energy sequence. Let $U_{n}$ be the potential \eqref{def:discrpot}. The following properties hold:
\begin{itemize}
\item[1)] For every $n\geq 1$, 
\begin{equation}\label{eq:orderdiscpot}
U_{n}(a_{n})\leq U_{n+1}(a_{n+1})\leq U_{n}(a_{n})+|a_{n+1}-a_{n}|^{\lambda}.
\end{equation}
\item[2)] For every $n\geq 1$, the function $U_{n}$ is not constant\footnotemark[4]\footnotetext[4]{See Remark~\ref{rmk:pot}.} on $S^{d}$.
\item[3)] For every $n\geq 1$ and $k\geq 0$, 
\begin{align}
n\,I_{\lambda}(\sigma_{d}) & < U_{n}(a_{n}) \leq n\,2^{\lambda}\label{ineq:boundpot}\\
U_{2k+1}(a_{2k+1}) & =U_{2k}(a_{2k})+2^{\lambda}\label{eq:equalpot}
\end{align}
understanding $U_{0}(a_{0})=0$.
\item[4)] We have
\begin{equation}\label{asympUnan}
\lim_{n\rightarrow\infty}\frac{U_{n}(a_{n})}{n}=I_{\lambda}(\sigma_{d}).
\end{equation} 
\end{itemize}
\end{theorem}
\begin{proof}
Let $n\geq 1$. We have $U_{n+1}(x)=U_{n}(x)+|x-a_{n}|^{\lambda}\geq U_{n}(x)$ for all $x\in S^{d}$, hence 
\[
U_{n}(a_{n})=\max_{x\in S^{d}}\,U_{n}(x)\leq \max_{x\in S^{d}}\,U_{n+1}(x)=U_{n+1}(a_{n+1}).
\] 
Taking $x=a_{n+1}$ in \eqref{eq:ineqpot}, we get
\[
U_{n}(a_{n})\geq U_{n}(a_{n+1})=\sum_{k=0}^{n}|a_{n+1}-a_{k}|^{\lambda}-|a_{n+1}-a_{n}|^{\lambda}=U_{n+1}(a_{n+1})-|a_{n+1}-a_{n}|^{\lambda}.
\]
So \eqref{eq:orderdiscpot} is justified.

Suppose that for some $n\geq 1$, the function $U_{n}(x)=\sum_{j=0}^{n-1}|x-a_{j}|^{\lambda}$ is constant on $S^{d}$. Then, one could construct a periodic greedy $\lambda$-energy sequence with period $n$, repeating indefinitely the cycle $(a_{0},\ldots,a_{n-1})$. But this contradicts \eqref{eq:asympdistr}. 

Integrating both sides of \eqref{eq:ineqpot} with respect to $\sigma_{d}$, and using the fact that the function $U_{n}(a_{n})-U_{n}(x)\geq 0$ is not constant on $S^{d}$, we get
\[
U_{n}(a_{n})>\int U_{n}(x)\,d\sigma_{d}(x)=\sum_{k=0}^{n-1}\int |x-a_{k}|^{\lambda}\,d\sigma_{d}(x)=\sum_{k=0}^{n-1}U^{\sigma_{d}}(a_{k})=n\,I_{\lambda}(\sigma_{d}),
\]
where we applied \eqref{eq:potconst} in the last equality. This proves the first inequality in \eqref{ineq:boundpot}. Clearly, $U_{1}(a_{1})=|a_{1}-a_{0}|^{\lambda}=2^{\lambda}$. The second inequality in \eqref{ineq:boundpot} then follows from \eqref{eq:orderdiscpot} and induction.

For each $k\geq 1$, applying \eqref{eq:symmpoteven} we get
\begin{align*}
U_{2k}(a_{2k}) & =U_{2k}(-a_{2k})=U_{2k}(a_{2k+1})\\ & =U_{2k+1}(a_{2k+1})-|a_{2k+1}-a_{2k}|^{\lambda} \\
& =U_{2k+1}(a_{2k+1})-2^{\lambda}.
\end{align*}
This justifies \eqref{eq:equalpot}.

Now we prove \eqref{asympUnan}. The argument we employ follows the line of reasoning in the proof of Theorem 1.2 in \cite[Ch. V]{ST}. By \eqref{ineq:boundpot} we have $\frac{U_{n}(a_{n})}{n}>I_{\lambda}(\sigma_{d})$ for all $n\geq 1$. Let $0<\epsilon<1$, and assume that $m\geq 1$ is a fixed index for which
\begin{equation}\label{eq:assumpm}
\frac{U_{m}(a_{m})}{m}\geq I_{\lambda}(\sigma_{d})+\epsilon.
\end{equation} 
From \eqref{eq:orderdiscpot} we deduce 
\begin{equation}\label{eq:iterpot}
U_{i}(a_{i})\geq U_{i+1}(a_{i+1})-2^{\lambda},\qquad i\geq 1.
\end{equation} 
Hence,
\[
\frac{U_{m-1}(a_{m-1})}{m}\geq \frac{U_{m}(a_{m})-2^{\lambda}}{m}\geq I_{\lambda}(\sigma_{d})+\epsilon-\frac{2^{\lambda}}{m}.
\]
Similarly, a repeated application of \eqref{eq:iterpot} yields
\[
\frac{U_{i}(a_{i})}{m}\geq I_{\lambda}(\sigma_{d})+\epsilon-(m-i)\frac{2^{\lambda}}{m},\qquad 1\leq i\leq m.
\]
In particular, we have
\begin{equation}\label{ineq:lowerbound}
\frac{U_{i}(a_{i})}{m}\geq I_{\lambda}(\sigma_{d})+\frac{\epsilon}{2},\qquad m(1-\frac{\epsilon}{2^{\lambda+1}})\leq i\leq m.
\end{equation}
Let $\kappa_{m}:=m(1-2^{-\lambda-1}\epsilon)$.

We have 
\begin{align*}
2^{-1}\,H_{\lambda}(\alpha_{m+1,\lambda}) & = \sum_{i=1}^{m}U_{i}(a_{i})\\
& = \sum_{1\leq i<\kappa_{m}}U_{i}(a_{i})+\sum_{\kappa_{m}\leq i\leq m} U_{i}(a_{i})\\
& \geq \sum_{1\leq i<\kappa_{m}} i\,I_{\lambda}(\sigma_{d})+\sum_{\kappa_{m}\leq i\leq m} m(I_{\lambda}(\sigma_{d})+\frac{\epsilon}{2})
\end{align*}
where we applied the first inequality in \eqref{ineq:boundpot} and \eqref{ineq:lowerbound}. This easily implies
\begin{equation}\label{eq:resultest}
\frac{H_{\lambda}(\alpha_{m+1,\lambda})}{(m+1)^{2}} \geq I_{\lambda}(\sigma_{d})\frac{(\lfloor \kappa_{m}\rfloor-1)\lfloor\kappa_{m}\rfloor}{(m+1)^{2}}+\left(I_{\lambda}(\sigma_{d})+\frac{\epsilon}{2}\right)\,\frac{2m(m-\lfloor\kappa_{m}\rfloor)}{(m+1)^{2}}
\end{equation}
where $\lfloor\cdot\rfloor$ is the floor function. Thus, we have shown that \eqref{eq:assumpm} implies \eqref{eq:resultest}. 

Assume that there are infinitely many indices $m$ that satisfy \eqref{eq:assumpm}, and they form the subsequence $\mathcal{N}$. Then, along this subsequence, the right-hand side of \eqref{eq:resultest} approaches the value
\[
I_{\lambda}(\sigma_{d})\left(1+\frac{\epsilon^{2}}{2^{2(\lambda+1)}}\right)+\frac{\epsilon^{2}}{2^{\lambda+1}}.
\] 
However, according to \eqref{eq:leadenergasymp} the left-hand side of \eqref{eq:resultest} has limit $I_{\lambda}(\sigma_{d})$ as $m\rightarrow\infty$ along $\mathcal{N}$. This contradiction shows that \eqref{eq:assumpm} is only valid for finitely many $m$'s, and concludes the proof of \eqref{asympUnan}.

Formula \eqref{asympUnan} can also be obtained as an application of \eqref{eq:asympdistr} and Theorem 1.2 in \cite{Sim}. It also follows from Theorem 2.1 in \cite{LopSaff}.
\end{proof}

\begin{remark}\label{rmk:pot}
Property 2) in Theorem~\ref{theo:asymppot} is not valid for $\lambda=2$. Indeed, we will see that in this case $U_{2k}\equiv 4k$ on $S^{d}$ for every $k\geq 1$.
\end{remark}

\subsection{Binary representation of the energy on the unit circle}

The rest of Section~\ref{sec:firstorder} is devoted to the analysis of greedy $\lambda$-energy sequences on the unit circle $S^{1}$.

\begin{lemma}\label{lem:geomstr:1}
Assume $0<\lambda\leq 1$, and let $z_{1}$ and $z_{2}$ be two distinct points on $S^{1}$. Consider the function 
\begin{equation}\label{eq:maxfunc}
|z-z_{1}|^{\lambda}+|z-z_{2}|^{\lambda}, \quad z\in S^{1}. 
\end{equation}
Let $\gamma$ be any of the two closed arcs on $S^{1}$ that connect $z_{1}$ and $z_{2}$. On $\gamma$, the function in \eqref{eq:maxfunc} has a unique maximum, which is attained at the middle point of the arc $\gamma$. 
\end{lemma}
\begin{proof}
Without loss of generality, we assume that $z_{1}=\overline{z_{2}}$. Hence, $z_{1}=e^{i\phi}$ and $z_{2}=e^{-i\phi}$, where $0<\phi<\pi$. Define the function 
\[
f(\theta):=|e^{i\theta}-e^{i\phi}|^{\lambda}+|e^{i\theta}-e^{-i\phi}|^{\lambda},\qquad -\phi\leq \theta\leq \phi.
\]
We have
\[
f(\theta)=2^{\lambda}\,g(\theta)
\]
where
\[
g(\theta):=\sin^{\lambda}\left(\frac{\phi-\theta}{2}\right)+\sin^{\lambda}\left(\frac{\phi+\theta}{2}\right),\qquad -\phi\leq \theta\leq \phi.
\]
The derivative of $g$ is 
\[
g'(\theta)=\frac{\lambda}{2}\,\left(\sin^{\lambda-1}\left(\frac{\phi+\theta}{2}\right)\cos\left(\frac{\phi+\theta}{2}\right)-\sin^{\lambda-1}\left(\frac{\phi-\theta}{2}\right)\cos\left(\frac{\phi-\theta}{2}\right)\right),\]
valid for $\theta\in(-\phi,\phi)$. Then,
\[
g''(\theta)=-\frac{\lambda}{4}\left[\sin^{\lambda-2}\left(\frac{\phi+\theta}{2}\right)\left(1-\lambda \cos^{2}\left(\frac{\phi+\theta}{2}\right)\right)+\sin^{\lambda-2}\left(\frac{\phi-\theta}{2}\right)\left(1-\lambda \cos^{2}\left(\frac{\phi-\theta}{2}\right)\right)\right],
\]
also for $\theta\in(-\phi,\phi)$.

Assume that $0<\lambda\leq 1$. Since $0<\frac{\phi\pm\theta}{2}<\pi$, we have $1-\lambda\cos^{2}(\frac{\phi\pm\theta}{2})>0$. We also have $\sin^{\lambda-2}((\phi\pm\theta)/2)>0$, therefore
\[
g''(\theta)<0 \qquad\mbox{for all}\,\,\theta\in(-\phi,\phi).
\]
It follows that $g$ and $f$ attain their maximum value on the interval $-\phi\leq \theta\leq \phi$ only at $\theta=0$. 
\end{proof}

It is interesting to remark that the property stated in Lemma~\ref{lem:geomstr:1} is not valid in the range $1<\lambda<2$. However, we have the following result due to Stolarsky, see \cite[Theorem 1.2]{Stol}.

\begin{lemma}[Stolarsky \cite{Stol}]\label{lem:Stolarsky}
Let $e_{1},\ldots,e_{n}$ be $n$ equally spaced points on $S^{1}$, and let $0<\lambda<2$. The function $\sum_{i=1}^{n}|z-e_{i}|^{\lambda}$, $z\in S^{1}$, attains its maximum value only at the midpoints of the arcs between consecutive $e_{i}$. 
\end{lemma}  

With Lemma~\ref{lem:Stolarsky} at hand, we can give a geometric description of greedy $\lambda$-energy sequences on $S^{1}$ for $0<\lambda<2$. We remind the reader that throughout this work, we always assume that greedy $\lambda$-energy sequences on $S^{1}$ have initial point $a_{0}=1$. 

We can use the same argument employed in the proof of \cite[Theorem 5]{BiaCal} to prove our next result. We reproduce the argument for convenience of the reader. In the rest of this work, we adopt the following notation. If $\alpha=(a_{0},\ldots,a_{k})$ and $\beta=(b_{0},\ldots,b_{l})$, then $(\alpha,\beta)$ will indicate the finite sequence $(a_{0},\ldots,a_{k},b_{0},\ldots,b_{l})$.

\begin{lemma}\label{lem:geomstr:2}
Let $0<\lambda<2$, and let $(a_n)_{n=0}^{\infty}$ be a greedy $\lambda$-energy sequence on $S^{1}$. For each $N\geq 1$, let $\alpha_{N}=(a_{0},\ldots,a_{N-1})$. Then, for all $n\geq 0$, the configuration $\alpha_{2^{n}}$ consists of the $2^{n}$-th roots of unity, and we have 
\begin{equation}\label{eq:section2n1}
\alpha_{2^{n+1}}=(\alpha_{2^{n}}, \rho \beta_{2^n}), 
\end{equation}
where $\rho$ is a solution of $z^{2^{n}}=-1$, and $\beta_{2^{n}}$ is the $2^{n}$-th section of a greedy $\lambda$-energy sequence on $S^{1}$ with initial point $b_{0}=1$.
\end{lemma}
\begin{proof}
We proceed by induction. For $n=0$, both statements are certainly true since $\alpha_{2}=(1,-1)$. Now assume that the $2^n$-th section of \emph{any} greedy $\lambda$-energy sequence on $S^{1}$ consists of the $2^n$-th roots of unity, and let us prove \eqref{eq:section2n1}. This, in turn, implies that the section $\alpha_{2^{n+1}}$ consists of the $2^{n+1}$-st roots of unity, which completes the induction argument. 

Set $m=2^n$. We want to prove that
\begin{equation}\label{eq:claimtp}
a_{m+k} = \rho b_k,\qquad 0\leq k\leq 2^n-1,
\end{equation}
where $\rho$ is a solution of $z^{2^{n}}=-1$ and $(b_{0},\ldots,b_{2^{n}-1})$ is the $2^{n}$-th section of a greedy sequence. By the induction hypothesis, $\alpha_{m}=(a_{0},\ldots,a_{m-1})$ consists of the $2^{n}$-th roots of unity. By Lemma~\ref{lem:Stolarsky}, the function  
\[
\sum_{\ell=0}^{m-1}|a_{\ell}-z|^{\lambda}
\]
is maximized precisely when $z=a_{m}$ is the midpoint of one of the arcs determined by two points in $\alpha_{m}$ lying consecutively on $S^{1}$. Therefore, the point $\rho:=a_{m}$ is a $2^{n}$-th root of $-1$. If we set $b_{0}=1$, then we have proved \eqref{eq:claimtp} for $k=0$. 

As a side remark, observe that if $0<\lambda\leq 1$, then we can apply Lemma~\ref{lem:geomstr:1} to justify that $\rho=a_{m}$ is a $2^{n}$-th root of $-1$. Indeed, if we assume, to fix notation, that $0\leq\arg(a_m)\leq 2\pi/m$, we  express the potential as follows
\begin{gather*}
\sum_{\ell=0}^{m-1} |a_{\ell}- z|^\lambda = \sum_{\ell=0}^{m-1} |e^{2\pi i \ell / m} - z|^\lambda \\
= \sum_{\ell=0}^{(m-2)/2} \left( |e^{-2 \pi i \ell / m} - z|^\lambda + |e^{2\pi i (1+\ell)/m} - z|^\lambda \right). 
\end{gather*}
By Lemma~\ref{lem:geomstr:1}, each expression $|e^{-2 \pi i \ell / m} - z |^\lambda + |e^{2\pi i (1+\ell)/m} - z|^\lambda$ has a unique maximum on the arc $\arg(z)\in[0,2\pi/m]$ at $\arg(z)=\pi/m$. Hence $\arg(a_{m})=\pi/m$ and $a_{m}=\rho$ is a $2^{n}$-th root of $-1$.  

We prove \eqref{eq:claimtp} by induction on $k$. Let $0\leq k<2^{n}-1$ be fixed. We assume that $a_{m+j}=\rho b_{j}$, $0\leq j\leq k$, where $(b_{0},\ldots,b_{k})$ is the $(k+1)$-st section of a greedy $\lambda$-energy sequence on $S^{1}$, and we prove that if $b_{k+1}$ is defined by $a_{m+k+1}=\rho b_{k+1}$, then
\begin{equation}\label{eq:tbpb}
\sum_{j=0}^{k}|b_{k+1}-b_{j}|^{\lambda}=\max_{|z|=1}\sum_{j=0}^{k}|z-b_{j}|^{\lambda}.
\end{equation}
Let $U_{n}(z)$ be the discrete potential \eqref{def:discrpot}. By definition, $a_{m+k+1}$ satisfies
\[ 
U_{m+k+1}(a_{m+k+1}) = \max_{|z|=1} U_{m+k+1}(z)
\]
and we have
\begin{align*}
U_{m+k+1}(a_{m+k+1}) & = U_{m+k+1}(\rho b_{k+1})=\max_{|z|=1} U_{m+k+1}(z)\\
& =\max_{|z|=1} U_{m+k+1}(\rho z)\\
& =\max_{|z|=1}(U_{m}(\rho z) + \sum_{j=0}^{k} |\rho z-\rho b_{j}|^\lambda)\\
& = \max_{|z|=1} (U_{m}(\rho z) + \sum_{j=0}^{k} |z-b_{j}|^{\lambda})\\
& =\max_{|z|=1} U_{m}(\rho z)+\max_{|z|=1}\sum_{j=0}^{k}|z-b_{j}|^{\lambda}
\end{align*}
where the last equality follows from the fact that both expressions $U_{m}(\rho z)$ and $\sum_{j=0}^{k}|z-b_{j}|^{\lambda}$ are maximized at a $2^{n}$-th root of unity. Therefore,
\begin{align*}
U_{m+k+1}(\rho b_{k+1}) & = U_{m}(\rho b_{k+1})+\sum_{j=0}^{k}|b_{k+1}-b_{j}|^{\lambda}\\
& =\max_{|z|=1}U_{m}(\rho z)+\max_{|z|=1}\sum_{j=0}^{k}|z-b_{j}|^{k}
\end{align*}
and \eqref{eq:tbpb} follows. This concludes the proof.\end{proof}

We deduce from Lemma~\ref{lem:geomstr:2} the following consequence. Assume $0<\lambda<2$ and $(a_{n})_{n=0}^{\infty}$ is a greedy $\lambda$-energy sequence on $S^{1}$. If $2^{n}$ is the largest power of $2$ that does not exceed $N$, then
\begin{equation}\label{eq:conseqgd}
\alpha_{N}=(\alpha_{2^{n}},\rho\beta_{N-2^{n}}),
\end{equation}
where $\rho$ satisfies $z^{2^{n}}=-1$ and $\beta_{N-2^{n}}$ is the section of order $N-2^{n}$ of a greedy $\lambda$-energy sequence on $S^{1}$.

Define
\begin{equation}\label{def:calLN}
\mathcal{L}_\lambda(N):=N\,2^\lambda \sum_{k=1}^{N-1} \left(\sin \frac{\pi k}{N}\right)^\lambda,\qquad N\geq 1.
\end{equation}
In the case $N=1$, we understand $\mathcal{L}_{\lambda}(1)=0$. The reader can check that this expression is the $\lambda$-energy of $N$ equally spaced points on $S^{1}$. In the range $0<\lambda<2$, this value is the largest $\lambda$-energy value for an $N$-point configuration on $S^{1}$, i.e.,
\begin{equation}\label{opteqspaced}
\mathcal{L}_{\lambda}(N)=\sup\{H_{\lambda}(\omega_{N}): \omega_{N}\subset S^{1}, \mbox{card}(\omega_{N})=N\},\qquad N\geq 2,
\end{equation}
see Theorem 2.3.3 and Remark 2.3.4 in \cite{BorHarSaff}. 

In our next result, we describe the $\lambda$-energy of the $N$-th section of a greedy $\lambda$-energy sequence, in terms of the binary representation of $N$.

\begin{lemma}\label{lem:energgeodesc}
Let $0<\lambda<2$, and let $(a_{n})_{n=0}^{\infty}$ be a greedy $\lambda$-energy sequence on $S^{1}$. Assume that $N\geq 2$ has the binary representation 
\begin{equation}\label{eq:Nbinrep}
N = 2^{n_1} + \cdots + 2^{n_p},\qquad  n_{1}>n_{2}>\cdots>n_{p}\geq 0. 
\end{equation}
Then, the $\lambda$-energy of $\alpha_N = (a_n)_{n=0}^{N-1}$ is given by
\begin{equation}\label{eq:energyalphaN}
H_{\lambda}(\alpha_N) = \sum_{k=1}^{p-1}\left(\sum_{j=k+1}^{p} 2^{n_j-n_k}\right)\mathcal{L}_{\lambda}(2^{n_k+1})
+ \sum_{k=1}^{p} \left(1-\sum_{j=k+1}^{p} 2^{n_j-n_k+1}\right)\mathcal{L}_\lambda(2^{n_k}),
\end{equation}
understanding $\sum_{i_1}^{i_2}$ as empty sum if $i_{2}<i_{1}$.
\end{lemma}
\begin{proof} We use induction on $p$, the number of terms in the binary expansion of $N$. If $p=1$, then $N=2^{n}$, so by Lemma~\ref{lem:geomstr:2} we have $H_{\lambda}(\alpha_{2^{n}})=\mathcal{L}_{\lambda}(2^{n})$, hence \eqref{eq:energyalphaN} holds in this case.

For a fixed $p\geq 1$, assume that \eqref{eq:energyalphaN} is valid for every integer $N\geq 2$ with $p$ terms in its binary expansion, and let $\widetilde{N}=2^{n_{1}}+\cdots+2^{n_{p}}+2^{n_{p+1}}$, $n_{1}>\cdots>n_{p}>n_{p+1}\geq 0$. Then, $2^{n_{1}}$ is the largest power of $2$ not greater than $N$, so by \eqref{eq:conseqgd}, we have
\[
\alpha_{\widetilde{N}}=(\alpha_{2^{n_1}},\rho \beta_{\widetilde{N}-2^{n_1}})
\]
where $\rho$ satisfies $z^{2^{n_{1}}}=-1$, and $\beta_{\widetilde{N}-2^{n_1}}$ is the section of order $\widetilde{N}-2^{n_{1}}=2^{n_{2}}+\cdots+2^{n_{p+1}}$ of a greedy $\lambda$-energy sequence. Let us write $A=\alpha_{2^{n_{1}}}$, $B=\beta_{\widetilde{N}-2^{n_1}}$.

Assume first that $\mathrm{card}(B)=\widetilde{N}-2^{n_{1}}\geq 2$. Since $A$ and $B$ are disjoint, we have 
\[
H_\lambda(\alpha_{\widetilde{N}}) = H_\lambda(A) + H_\lambda(B) + 2 \sum_{y \in B} \sum_{x \in A} |x - y|^\lambda.
\]
The configuration $A$ is formed by equally spaced points, so
\[
H_{\lambda}(A)=\mathcal{L}_{\lambda}(2^{n_{1}}).
\]
We also know by Lemma~\ref{lem:geomstr:2} that every point $y \in B$ is a midpoint of an arc whose endpoints are neighboring points in $A$. Therefore, for each $y\in B$, we have
\[
\sum_{x\in A}|x-y|^{\lambda}=2^{-n_{1}-1}\mathcal{L}_{\lambda}(2^{n_{1}+1})-2^{-n_{1}}\mathcal{L}_{\lambda}(2^{n_{1}}),
\]
consequently,
\begin{align*}
2 \sum_{y \in B} \sum_{x \in A} |x - y|^\lambda
& = 2\,\mathrm{card}(B) \left( 2^{-n_{1}-1}\mathcal{L}_{\lambda}(2^{n_{1}+1})-2^{-n_{1}}\mathcal{L}_{\lambda}(2^{n_{1}})\right)\\
& = \sum_{j=2}^{p+1} 2^{n_j - n_1} \mathcal{L}_\lambda(2^{n_1+1}) - \sum_{j=2}^{p+1} 2^{n_j - n_1 + 1} \mathcal{L}_\lambda(2^{n_1}).
\end{align*}
We obtain
\begin{align*}
H_{\lambda}(\alpha_{\widetilde{N}}) = \left(\sum_{j=2}^{p+1} 2^{n_j - n_1}\right) \mathcal{L}_{\lambda}(2^{n_1+1})
+ \left(1 - \sum_{j=2}^{p+1} 2^{n_j - n_1 + 1}\right) \mathcal{L}_{\lambda}(2^{n_1}) + H_{\lambda}(B).
\end{align*}
By the induction hypothesis, the energy of $B$ is given by
\begin{align*}
H_{\lambda}(B) = \sum_{k=2}^{p}\left(\sum_{j=k+1}^{p+1} 2^{n_j-n_k}\right)\mathcal{L}_{\lambda}(2^{n_k+1})
+ \sum_{k=2}^{p+1}\left(1 - \sum_{j=k+1}^{p+1} 2^{n_j-n_k+1}\right)\mathcal{L}_{\lambda}(2^{n_k}).
\end{align*}
From the previous two expressions we deduce 
\[
H_{\lambda}(\alpha_{\widetilde{N}})=\sum_{k=1}^{p}\left(\sum_{j=k+1}^{p+1} 2^{n_j-n_k}\right)\mathcal{L}_{\lambda}(2^{n_k+1})
+ \sum_{k=1}^{p+1}\left(1 - \sum_{j=k+1}^{p+1} 2^{n_j-n_k+1}\right)\mathcal{L}_{\lambda}(2^{n_k}),
\]
which proves \eqref{eq:energyalphaN} for $\widetilde{N}$, in the case $\mathrm{card}(B)=\widetilde{N}-2^{n_{1}}\geq 2$. If $\mathrm{card}(B)=1$, then $\widetilde{N}=2^{n_{1}}+1$, $n_{1}\geq 1$, so
\begin{align*}
H_{\lambda}(\alpha_{\widetilde{N}}) & =H_{\lambda}(A)+2\sum_{x\in A}|x-\rho|^{\lambda}\\
& =2^{-n_{1}}\mathcal{L}_{\lambda}(2^{n_1+1})+(1-2^{-n_{1}+1}) \mathcal{L}_{\lambda}(2^{n_1})
\end{align*}
which coincides with \eqref{eq:energyalphaN}.
\end{proof}

We define now the quantities
\begin{equation}\label{def:UcalN}
\mathcal{U}_{\lambda}(N):=\sum_{k=0}^{N-1}\left|e^{i\pi/N} - e^{2\pi i k/N}\right|^{\lambda},\quad N\geq 1.
\end{equation}
This represents the discrete potential of $N$ equally spaced points evaluated at the midpoint of one of the arcs between two adjacent points. If $e_{1},\ldots,e_{N}$ are $N$ equally spaced points on the unit circle, then
\[
\sum_{i=1}^{N-1}|e_{i}-e_{N}|^{\lambda}=\frac{\mathcal{L}_{\lambda}(N)}{N}.
\]
From this relation it is easy to deduce the identity
\begin{equation}\label{eq:relUnLn}
\mathcal{U}_{\lambda}(N)=\frac{\mathcal{L}_{\lambda}(2N)}{2N}-\frac{\mathcal{L}_{\lambda}(N)}{N}.
\end{equation}

\begin{lemma}\label{lem:discrpotgeomdesc}
Let $0<\lambda<2$, and let $(a_{n})_{n=0}^{\infty}$ be a greedy $\lambda$-energy sequence on $S^{1}$. If $N\geq 1$ has the binary representation \eqref{eq:Nbinrep}, then the function $U_{N}(x)$ \eqref{def:discrpot} satisfies 
\begin{equation}\label{eq:descUnan}
U_{N}(a_{N})=\sum_{k=1}^{p}\mathcal{U}_{\lambda}(2^{n_{k}}).
\end{equation}
\end{lemma}
\begin{proof}
The proof is also by induction on $p$. If $p=1$ and $N=2^{n}$, then by Lemma~\ref{lem:geomstr:2} the points in $\alpha_{N}=(a_{0},\ldots,a_{N-1})$ are the $N$-th roots of unity. The point $a_{N}$ is the midpoint of one of the arcs between adjacent points in $\alpha_{N}$, hence $U_{N}(a_{N})=\mathcal{U}_{\lambda}(N)$.

For a fixed $p\geq 1$, assume as induction hypothesis that \eqref{eq:descUnan} is valid for every greedy sequence and every $N\geq 1$ with binary representation of length $p$. Let $\widetilde{N}=2^{n_{1}}+\cdots+2^{n_{p}}+2^{n_{p+1}}$, $n_{1}>\cdots>n_{p}>n_{p+1}\geq 0$. Then, as in the proof of Lemma~\ref{lem:energgeodesc}, we can write
\[
\alpha_{\widetilde{N}}=(\alpha_{2^{n_{1}}},\rho \beta_{\widetilde{N}-2^{n_1}}),
\]
where $\rho$ satisfies $z^{2^{n_{1}}}=-1$, and $\beta_{\widetilde{N}-2^{n_1}}$ is the section of order $\widetilde{N}-2^{n_{1}}$ of a greedy $\lambda$-energy sequence $(b_{k})_{k=0}^{\infty}$. So $a_{k}=\rho\,b_{k-2^{n_1}}$, $2^{n_{1}}\leq k\leq \widetilde{N}$, where $(b_{0},\ldots,b_{\widetilde{N}-2^{n_1}})=\beta_{\widetilde{N}-2^{n_1}+1}$. We have
\[
U_{\widetilde{N}}(a_{\widetilde{N}})=U_{2^{n_1}}(a_{\widetilde{N}})+\sum_{k=2^{n_1}}^{\widetilde{N}-1}|a_{k}-a_{\widetilde{N}}|^{\lambda}.
\]
The point $a_{\widetilde{N}}$ is the midpoint of one of the arcs between two adjacent points in $\alpha_{2^{n_1}}$, hence 
\[
U_{2^{n_1}}(a_{\widetilde{N}})=\mathcal{U}_{\lambda}(2^{n_1}).
\] 
We also have
\[
\sum_{k=2^{n_1}}^{\widetilde{N}-1}|a_{k}-a_{\widetilde{N}}|^{\lambda}=
\sum_{k=2^{n_1}}^{\widetilde{N}-1}|\rho\, b_{k-2^{n_1}}-\rho\, b_{\widetilde{N}-2^{n_1}}|^{\lambda}=\sum_{k=0}^{\widetilde{N}-2^{n_1}-1}|b_{k}-b_{\widetilde{N}-2^{n_1}}|^{\lambda}=U_{M}(b_{M}),
\]
where $M:=\widetilde{N}-2^{n_{1}}$. The number $M=2^{n_{2}}+\cdots+2^{n_{p+1}}$ has a binary representation of length $p$, so by induction hypothesis we obtain
\[
U_{M}(b_{M})=\sum_{k=2}^{p+1}\mathcal{U}_{\lambda}(2^{n_{k}}).
\]
In conclusion, $U_{\widetilde{N}}(a_{\widetilde{N}})=\sum_{k=1}^{p+1}\mathcal{U}_{\lambda}(2^{n_k})$, which finishes the proof of \eqref{eq:descUnan}.
\end{proof}

We can write a more convenient expression for $\mathcal{U}_{\lambda}(2^{n})$. From \eqref{eq:relUnLn} and \eqref{def:calLN} we get
\[
\mathcal{U}_{\lambda}(2^{n})=\frac{\mathcal{L}_{\lambda}(2^{n+1})}{2^{n+1}}-\frac{\mathcal{L}_{\lambda}(2^{n})}{2^{n}}
=2^{\lambda} \sum_{k=0}^{2^{n}-1} \sin^{\lambda}\left(\frac{(2k+1)\pi}{2^{n+1}}\right).
\]

\subsection{Second-order asymptotics on the unit circle}

Let $\zeta(s):=\sum_{n=1}^{\infty}\frac{1}{n^{s}}$, $\mathrm{Re}(s)>1$, be the classical Riemann zeta function. We also use $\zeta(s)$ to denote the meromorphic continuation of this function to the entire complex plane. Recall that this function has a simple pole at $s=1$, it is analytic on $\mathbb{C}\setminus\{1\}$, takes negative values on the interval $(-2,1)$, and it has trivial zeros at the points $s=-2n$, $n\geq 1$.

Recall the expression
\[
I_{\lambda}(\sigma_{1})=\frac{2^{\lambda}\,\Gamma((1+\lambda)/2)}{\sqrt{\pi}\,\Gamma(1+\lambda/2)}.
\]  

\begin{lemma}\label{lem:asympUcal}
Let $0<\lambda<2$. For the sequence \eqref{def:UcalN} we have
\begin{align}
\lim_{N\rightarrow\infty}\frac{\mathcal{U}_{\lambda}(N)}{N} & =I_{\lambda}(\sigma_{1}),\label{eq:asympUcalN:1}\\
\lim_{N\rightarrow\infty}\frac{\mathcal{U}_{\lambda}(N)-N I_{\lambda}(\sigma_{1})}{N^{-\lambda}} & =(2^{-\lambda}-1)\,(2\pi)^{\lambda}\,2\zeta(-\lambda).\label{eq:asympUcalN:2}
\end{align}
\end{lemma}
\begin{proof}
The following asymptotic formula, valid for $0<\lambda<2$, follows from \cite[Theorem 1.1]{BrauHardSaff} (take $s=-\lambda$ and $p=0$ in that result):
\begin{equation}\label{eq:asympexpLN}
\mathcal{L}_{\lambda}(N)=I_{\lambda}(\sigma_{1}) N^2+ (2\pi)^{\lambda}\,2\zeta(-\lambda)\,N^{1-\lambda}+\mathcal{O}_{\lambda}(N^{-1-\lambda}),\qquad N\rightarrow\infty.
\end{equation}
Applying \eqref{eq:relUnLn} and the first two terms in \eqref{eq:asympexpLN}, we obtain \eqref{eq:asympUcalN:1}:
\begin{align*}
\lim_{N\rightarrow\infty}\frac{\mathcal{U}_{\lambda}(N)}{N}=\lim_{N\rightarrow\infty} \left(2\,\frac{\mathcal{L}_{\lambda}(2N)}{(2N)^2}-\frac{\mathcal{L}_{\lambda}(N)}{N^2}\right)=2 I_{\lambda}(\sigma_{1})-I_{\lambda}(\sigma_{1})=I_{\lambda}(\sigma_{1}).
\end{align*}
Similarly, we can write
\[
\frac{\mathcal{U}_{\lambda}(N)-N I_{\lambda}(\sigma_{1})}{N^{-\lambda}}
=2^{-\lambda}\,\frac{\mathcal{L}_{\lambda}(2N)-(2N)^2 I_{\lambda}(\sigma_{1})}{(2N)^{1-\lambda}}-\frac{\mathcal{L}_{\lambda}(N)-N^2\,I_{\lambda}(\sigma_{1})}{N^{1-\lambda}}.
\]
Letting $N\rightarrow\infty$ and using \eqref{eq:asympexpLN}, we get \eqref{eq:asympUcalN:2}.\end{proof}

\begin{theorem}\label{theo:normpot}
Let $0<\lambda<2$, and let $(a_{n})_{n=0}^{\infty}\subset S^{1}$ be a greedy $\lambda$-energy sequence. Then, the sequence $(U_{N}(a_{N})-N I_{\lambda}(\sigma_{1}))_{N=1}^{\infty}$ is bounded and divergent. For each $N\geq 1$,
\begin{equation}\label{eq:boundUnansecord}
0<U_{N}(a_{N})-N I_{\lambda}(\sigma_{1})<I_{\lambda}(\sigma_{1}).\end{equation}
These inequalities are sharp since we have
\begin{align}\liminf_{N\rightarrow\infty}\,(U_{N}(a_{N})-N I_{\lambda}(\sigma_{1})) & =0,\label{eq:secordUnanliminf}\\
\limsup_{N\rightarrow\infty}\,(U_{N}(a_{N})-N I_{\lambda}(\sigma_{1})) & =I_{\lambda}(\sigma_{1}).
\label{eq:secordUnanlimsup}
\end{align}
\end{theorem}
\begin{proof}
We know from \eqref{ineq:boundpot} that $U_{N}(a_{N})-N I_{\lambda}(\sigma_{1})>0$ for all $N\geq 1$. This is the first inequality in \eqref{eq:boundUnansecord}. 

If we take $N=2^{n}$, $n\rightarrow\infty$, and apply \eqref{eq:asympUcalN:2}, we get
\[
\lim_{n\rightarrow\infty}(U_{2^{n}}(a_{2^{n}})-2^{n}\,I_{\lambda}(\sigma_{1}))=\lim_{n\rightarrow\infty}(\mathcal{U}_{\lambda}(2^{n})-2^{n}\,I_{\lambda}(\sigma_{1}))=0.
\] 
This and the first inequality in \eqref{eq:boundUnansecord} imply \eqref{eq:secordUnanliminf}. 

Consider the subsequence
\[
\kappa_{n}:=\sum_{k=0}^{n-1} 2^{k}=2^{n}-1,\qquad n\geq 1. 
\]
We claim that the sequence $(U_{\kappa_{n}}(a_{\kappa_{n}})-\kappa_{n}\,I_{\lambda}(\sigma_1))_{n=1}^{\infty}$ is strictly increasing. Indeed, according to Lemma~\ref{lem:discrpotgeomdesc}, we have
\begin{align*}
U_{\kappa_{n+1}}(a_{\kappa_{n+1}})-\kappa_{n+1}\,I_{\lambda}(\sigma_1) & =\sum_{k=0}^{n}(\mathcal{U}_{\lambda}(2^{k})-2^{k} I_{\lambda}(\sigma_{1}))\\
& =U_{\kappa_{n}}(a_{\kappa_{n}})-\kappa_{n}\,I_{\lambda}(\sigma_1)+\mathcal{U}_{\lambda}(2^{n})-2^{n} I_{\lambda}(\sigma_{1})
\end{align*}
and $\mathcal{U}_{\lambda}(2^{n})-2^{n} I_{\lambda}(\sigma_{1})=U_{2^{n}}(a_{2^{n}})-2^{n} I_{\lambda}(\sigma_{1})>0$ for each $n\geq 1$. 

Using \eqref{eq:relUnLn}, we obtain
\[
U_{\kappa_{n}}(a_{\kappa_{n}})=\sum_{k=0}^{n-1}\mathcal{U}_{\lambda}(2^{k})=\sum_{k=0}^{n-1}\left(\frac{\mathcal{L}_{\lambda}(2^{k+1})}{2^{k+1}}-\frac{\mathcal{L}_{\lambda}(2^{k})}{2^{k}}\right)=\frac{\mathcal{L}_{\lambda}(2^{n})}{2^{n}},
\]
(recall that $\mathcal{L}_{\lambda}(1)=0$) so 
\begin{align*}
U_{\kappa_{n}}(a_{\kappa_{n}})-\kappa_{n}\,I_{\lambda}(\sigma_1) & =\frac{\mathcal{L_{\lambda}}(2^{n})}{2^{n}}-2^{n}\,I_{\lambda}(\sigma_{1})+I_{\lambda}(\sigma_{1})\\
& =\frac{\mathcal{L}_{\lambda}(2^{n})-2^{2n} I_{\lambda}(\sigma_1)}{2^{n}}+I_{\lambda}(\sigma_{1}).
\end{align*}
In virtue of \eqref{eq:asympexpLN}, we have $\lim_{N\rightarrow\infty} (\mathcal{L}_{\lambda}(N)-N^2\,I_{\lambda}(\sigma_{1}))/N=0$, hence the previous identity gives
\begin{equation}\label{eq:limUkappa}
\lim_{n\rightarrow\infty}(U_{\kappa_{n}}(a_{\kappa_{n}})-\kappa_{n}\,I_{\lambda}(\sigma_{1}))=I_{\lambda}(\sigma_{1}).
\end{equation}
By monotonicity, this implies
\[
U_{\kappa_{n}}(a_{\kappa_{n}})-\kappa_{n}\,I_{\lambda}(\sigma_{1})<I_{\lambda}(\sigma_{1}),\qquad \mbox{for all}\,\,n\geq 1.
\]
If $N$ has the binary expansion \eqref{eq:Nbinrep}, applying \eqref{eq:descUnan} we obtain
\begin{align*}
U_{N}(a_{N})-N I_{\lambda}(\sigma_{1}) & =\sum_{k=1}^{p}(\mathcal{U}_{\lambda}(2^{n_{k}})-2^{n_{k}} I_{\lambda}(\sigma_1))\\ 
& \leq \sum_{j=0}^{n_{1}}(\mathcal{U}_{\lambda}(2^{j})-2^{j} I_{\lambda}(\sigma_1))\\
& =U_{\kappa_{n}}(a_{\kappa_{n}})-\kappa_{n} I_{\lambda}(\sigma_{1})\\
& <I_{\lambda}(\sigma_1),
\end{align*}
where $n=n_{1}+1$. This justifies the second inequality in \eqref{eq:boundUnansecord}.

Finally, \eqref{eq:secordUnanlimsup} follows from \eqref{eq:boundUnansecord} and \eqref{eq:limUkappa}.
\end{proof}

\begin{remark}
Interestingly, the two very close sequences $(2^{n})$ and $(2^{n}-1)$ provide the liminf and limsup of the sequence $(U_{N}(a_{N})-N I_{\lambda}(\sigma_1))$, respectively.
\end{remark}

In Figures~\ref{plotsecord5} and \ref{plotsecord6} we present plots of the sequences analyzed in Theorem~\ref{theo:normpot}, for different values of $\lambda\in(0,2)$.

\begin{figure}
\centering
\begin{subfigure}[b]{0.475\textwidth}
\centering
\includegraphics[width=\textwidth]{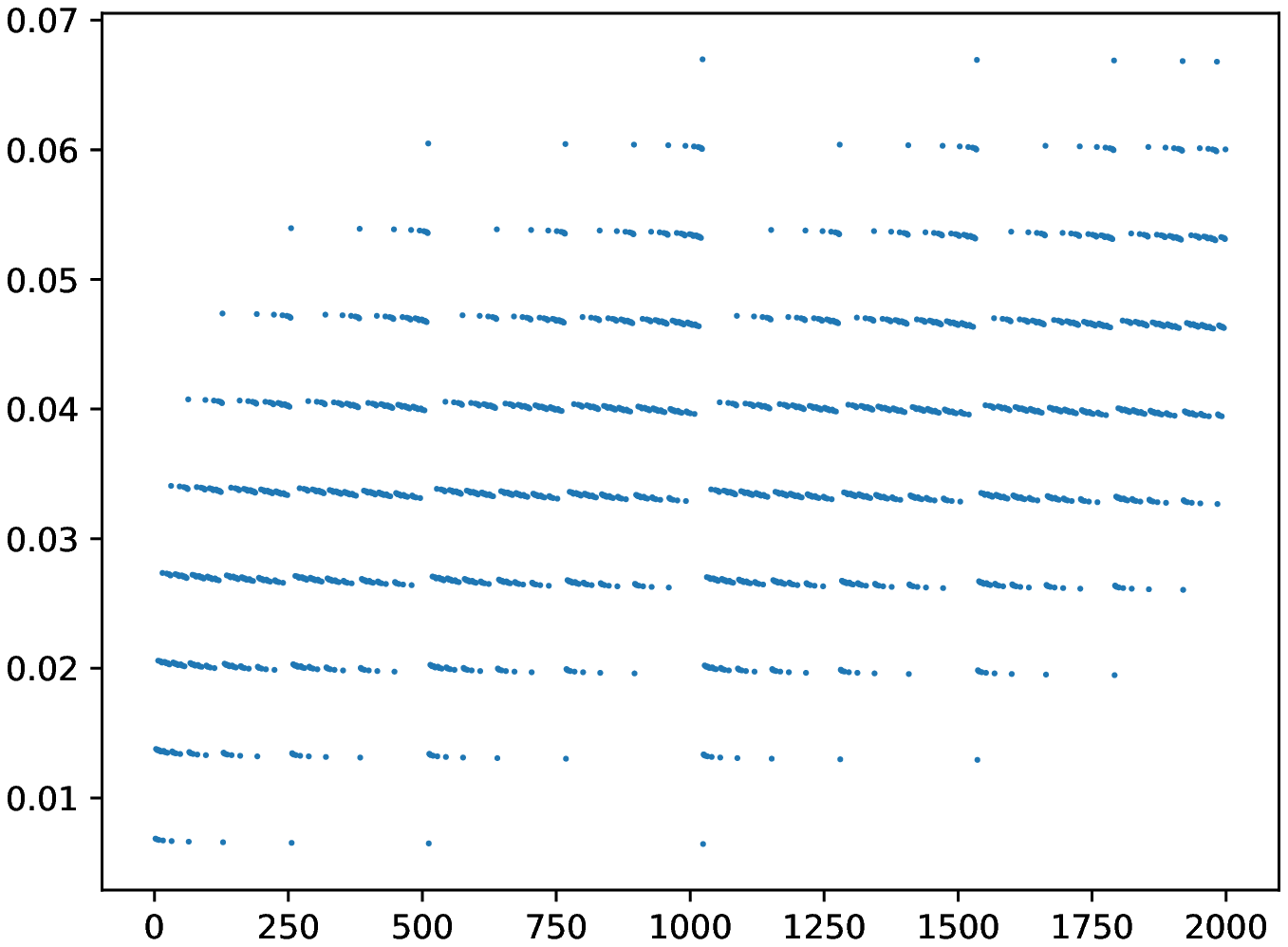}
\caption[]{{$\lambda = 0.01$}}    
\end{subfigure}
\hfill
\begin{subfigure}[b]{0.475\textwidth}  
\centering 
\includegraphics[width=\textwidth]{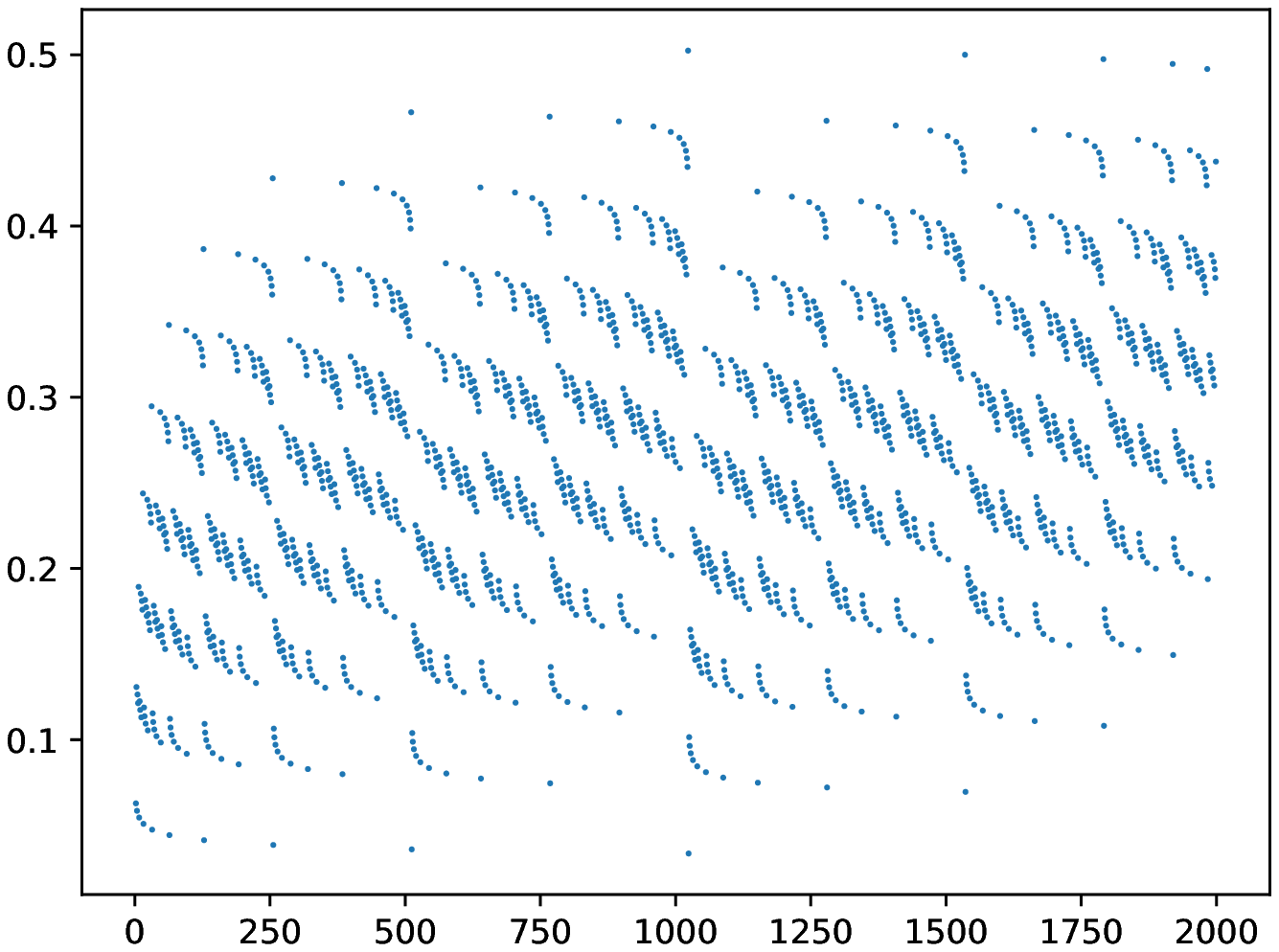}
\caption[]{{$\lambda = 0.1$}}    
\end{subfigure}
\vskip\baselineskip
\begin{subfigure}[b]{0.475\textwidth}   
\centering 
\includegraphics[width=\textwidth]{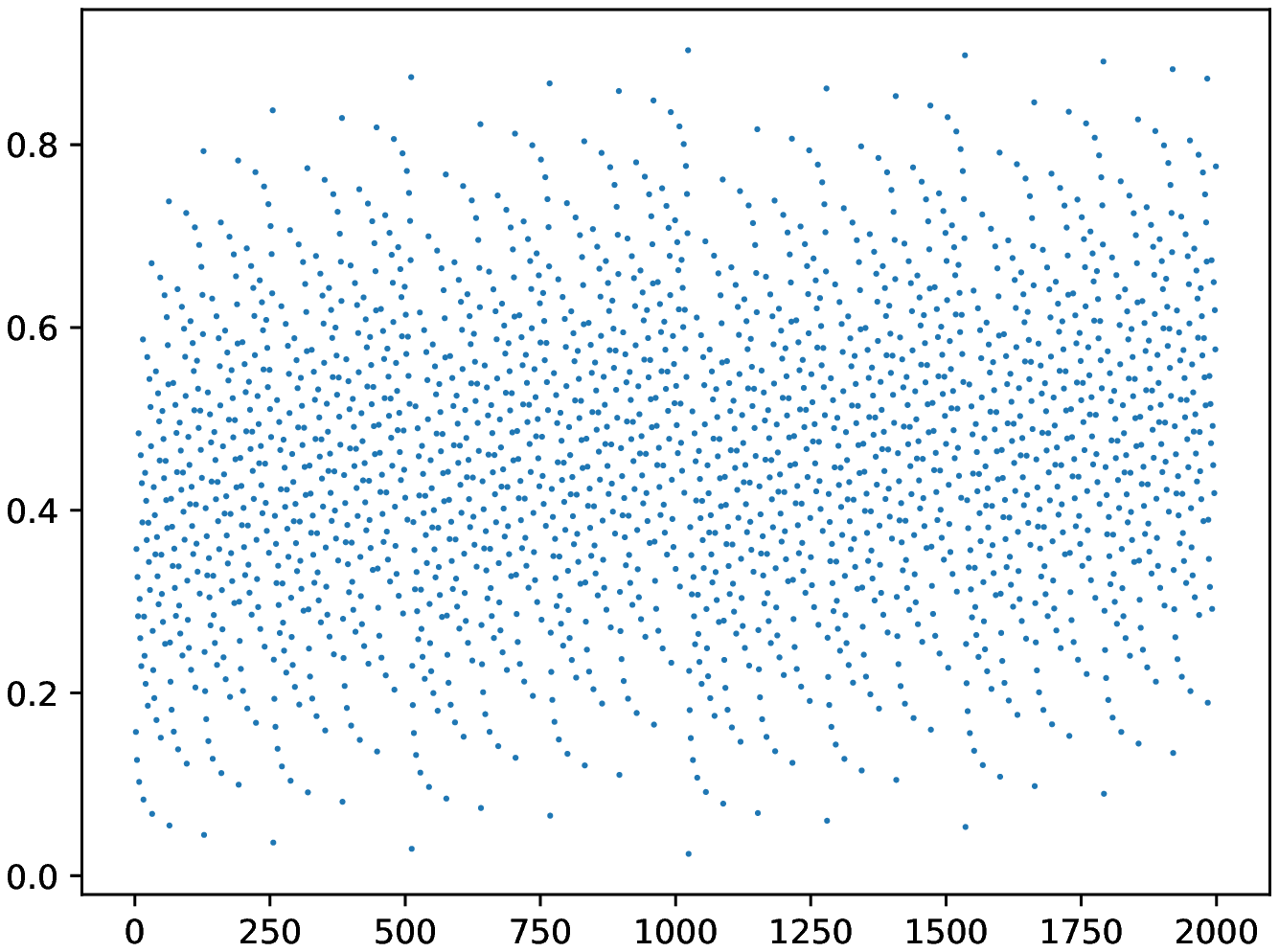}
\caption[]{{$\lambda = 0.3$}}
\end{subfigure}
\quad
\begin{subfigure}[b]{0.475\textwidth}   
\centering 
\includegraphics[width=\textwidth]{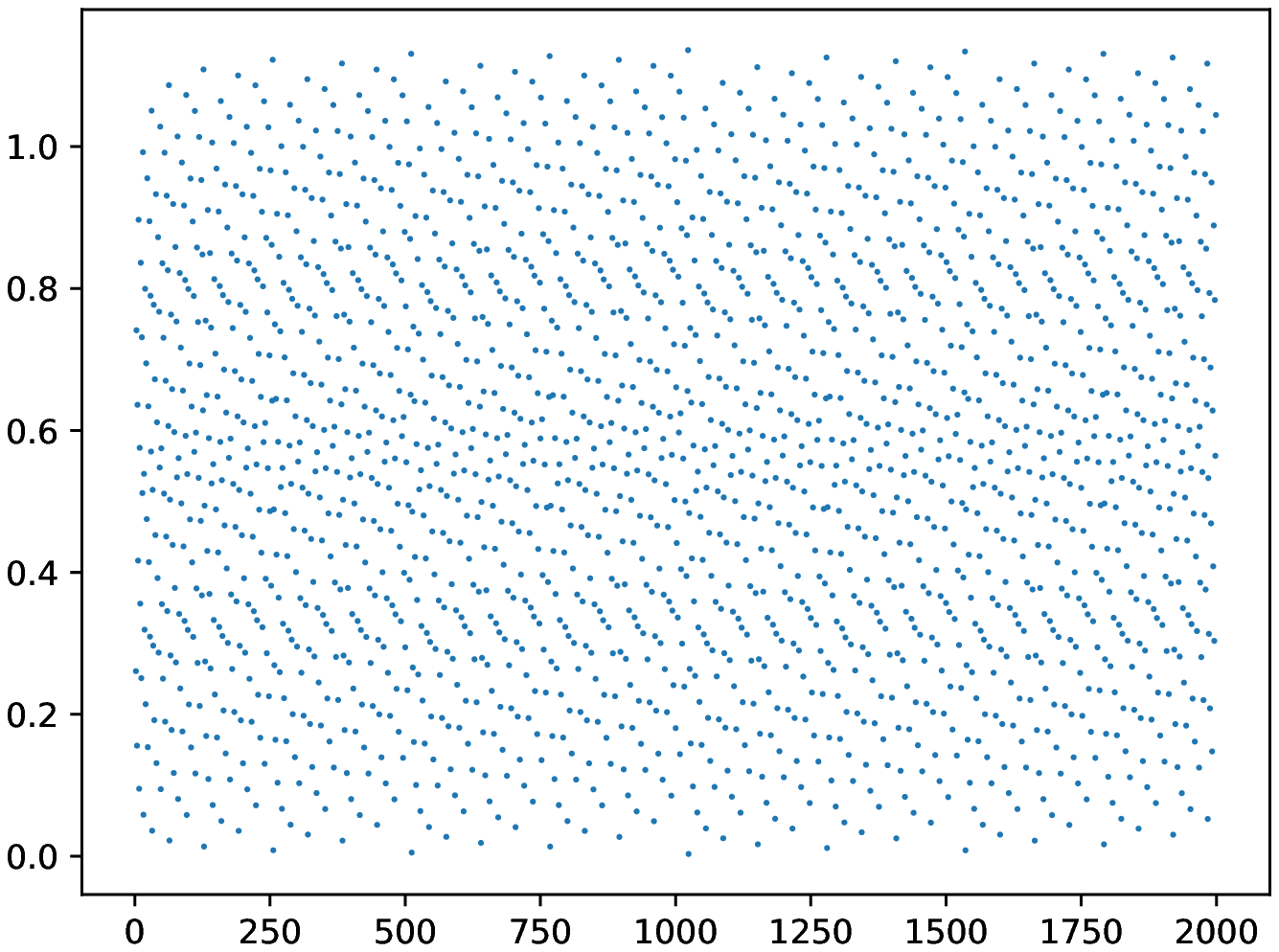}
\caption[]{{$\lambda = 0.7$}}
\end{subfigure}
\caption[]{{Plots of sequences $(U_{N}(a_{N})-N I_{\lambda}(\sigma_{1}))$ for $0<\lambda<1$ and $1\leq N\leq 2000$.}}
\label{plotsecord5}
\end{figure}

\begin{figure}
\centering
\begin{subfigure}[b]{0.475\textwidth}
\centering
\includegraphics[width=\textwidth]{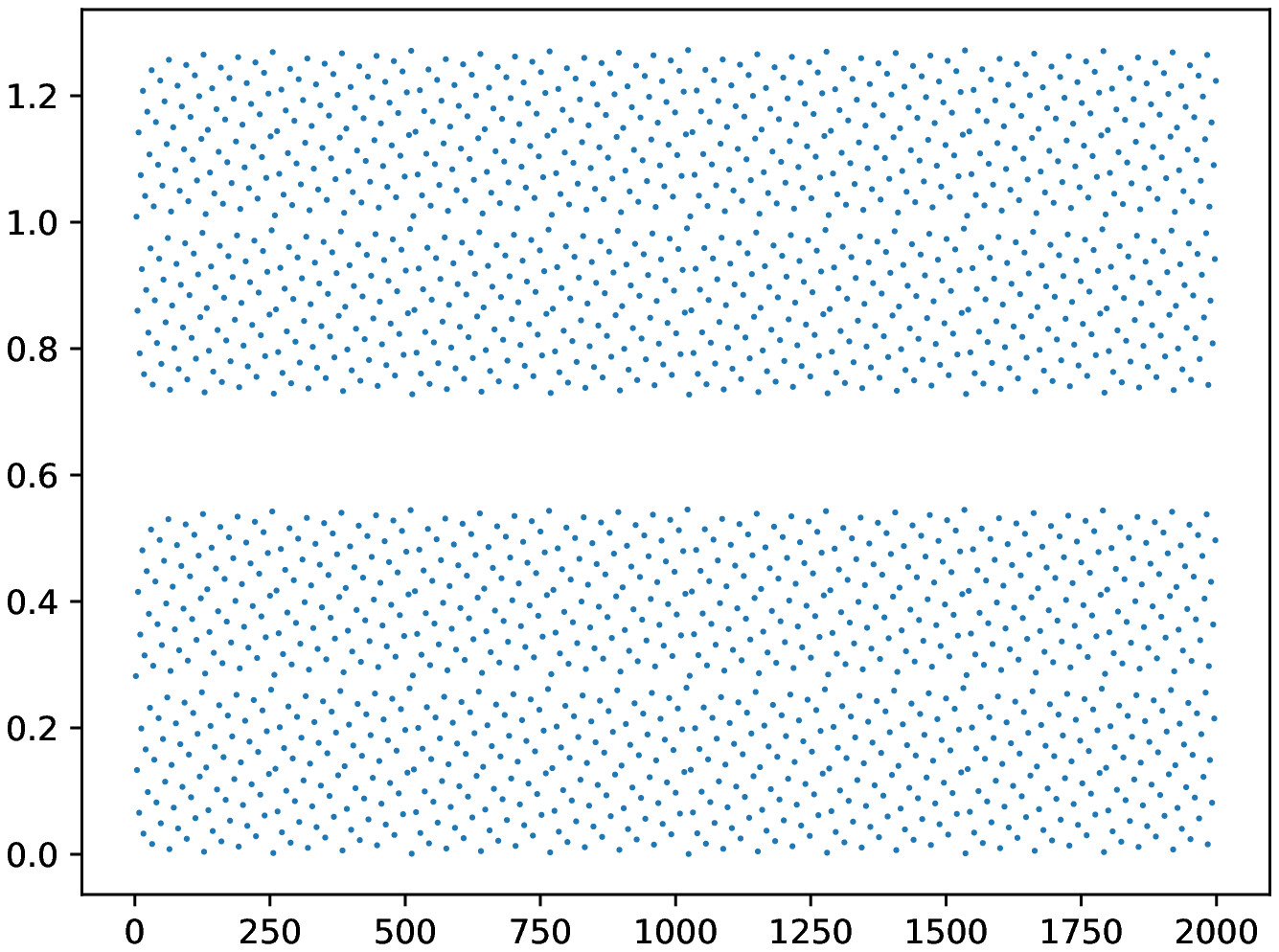}
\caption[]{{$\lambda = 1$}}
\end{subfigure}
\hfill
\begin{subfigure}[b]{0.475\textwidth}  
\centering 
\includegraphics[width=\textwidth]{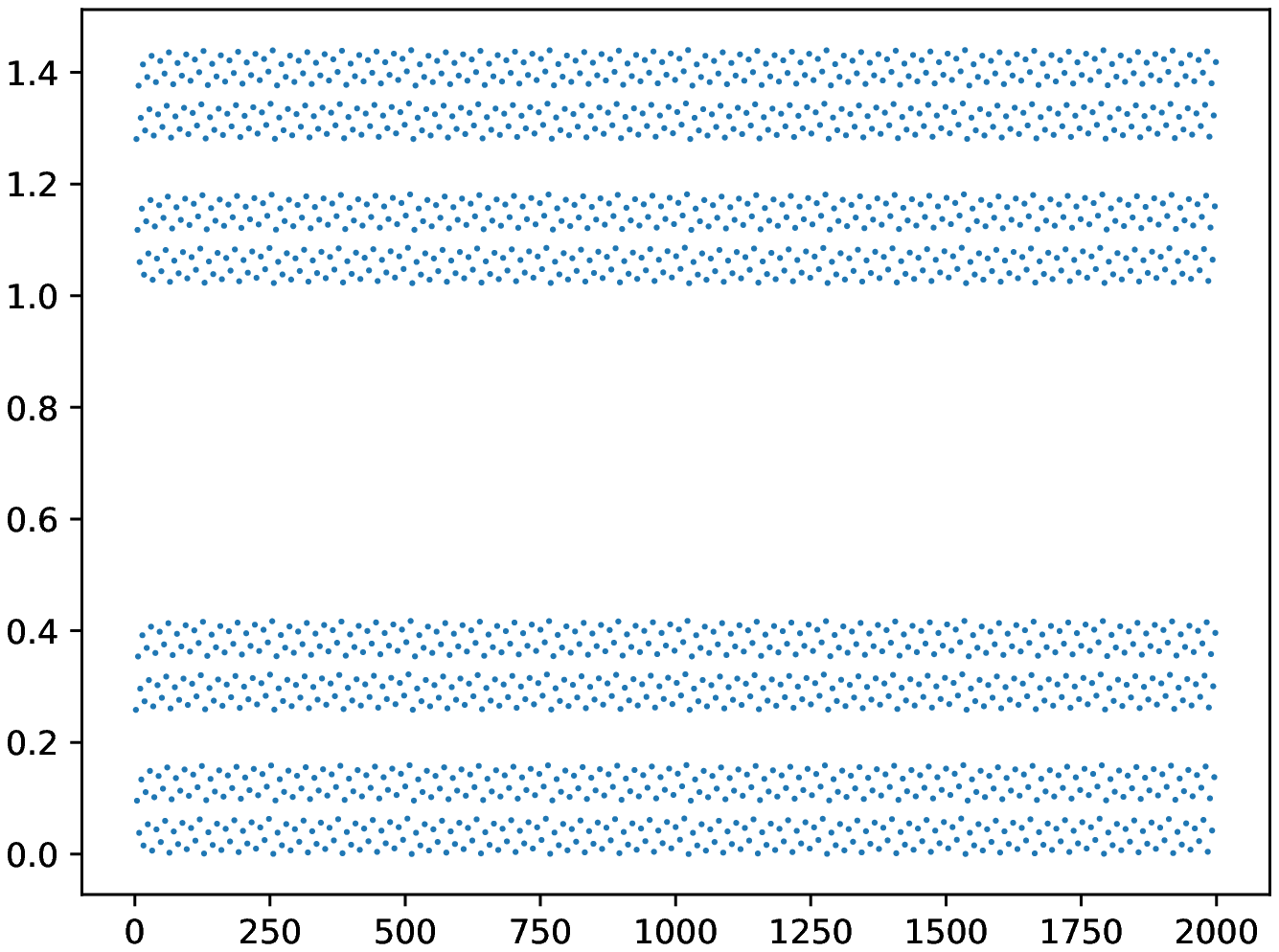}
\caption[]{{$\lambda = 1.3$}}
\end{subfigure}
\vskip\baselineskip
\begin{subfigure}[b]{0.475\textwidth}   
\centering 
\includegraphics[width=\textwidth]{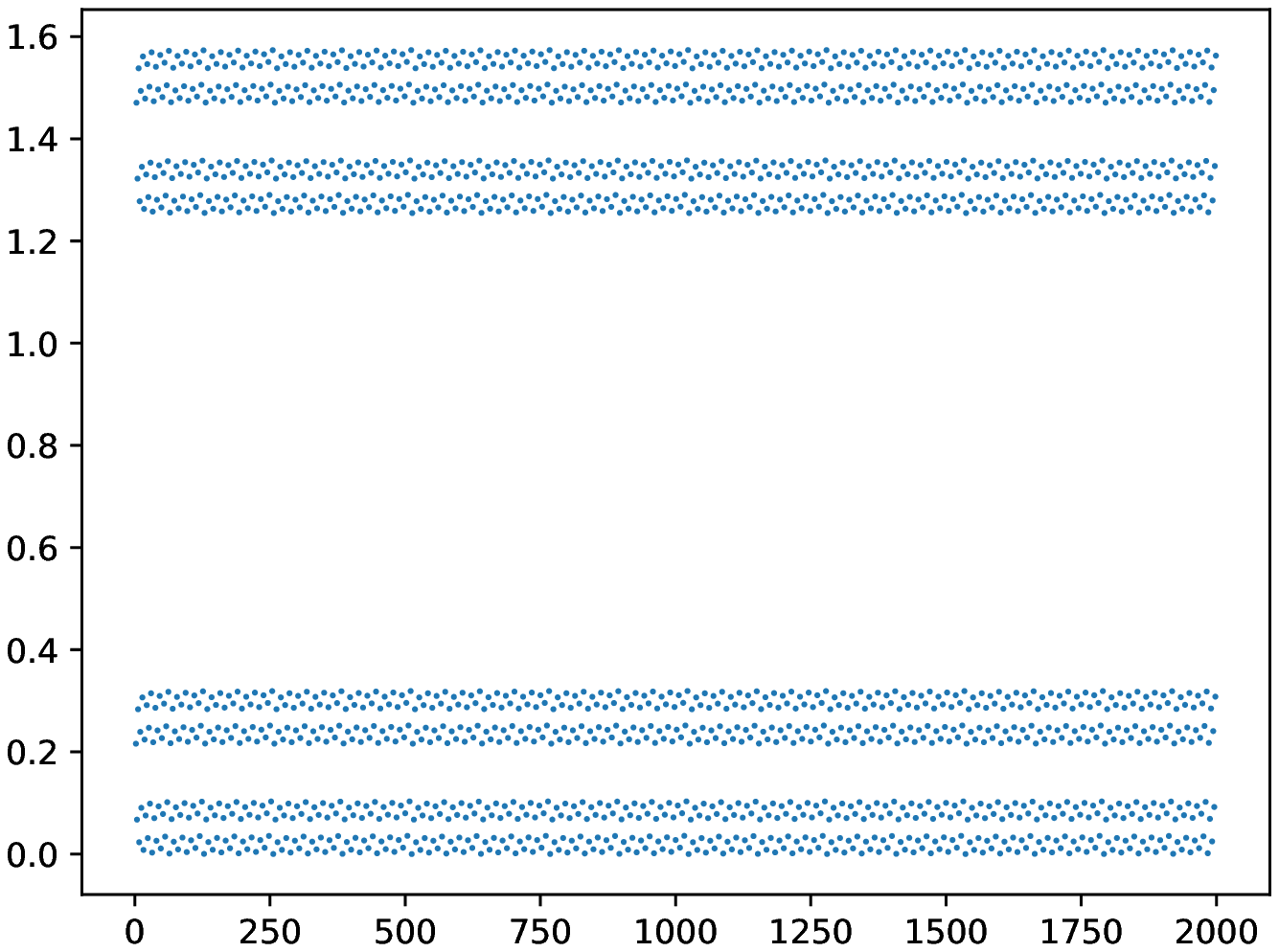}
\caption[]{{$\lambda = 1.5$}}
\end{subfigure}
\quad
\begin{subfigure}[b]{0.475\textwidth}   
\centering 
\includegraphics[width=\textwidth]{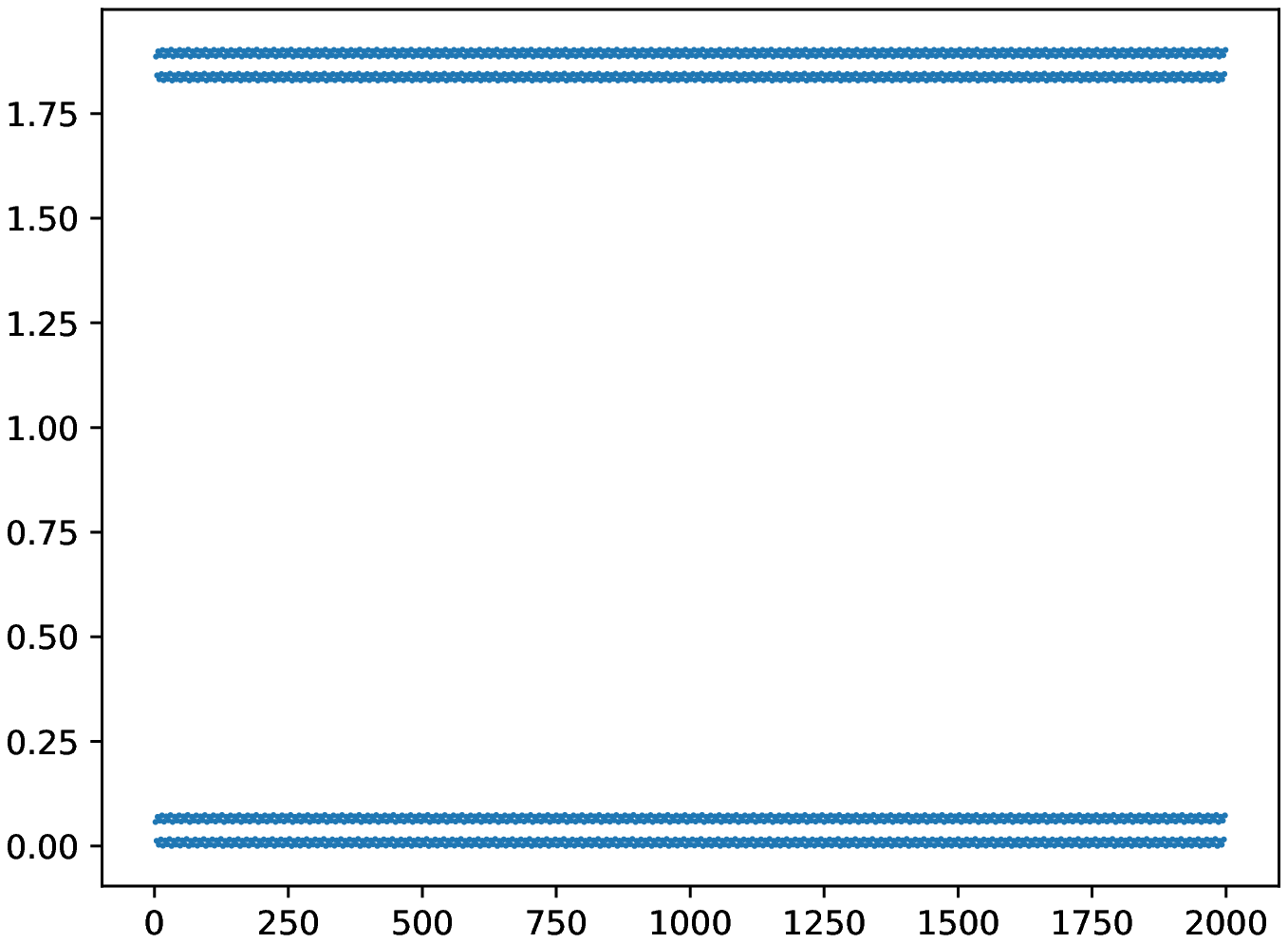}
\caption[]{{$\lambda = 1.9$}} 
\end{subfigure}
\caption[]{{Plots of sequences $(U_{N}(a_{N})-N I_{\lambda}(\sigma_{1}))$ for $1\leq\lambda<2$ and $1\leq N\leq2000$.}}
\label{plotsecord6}
\end{figure}

Our goal now is the study of the second-order asymptotics of the energy $H_{\lambda}(\alpha_{N,\lambda})$ in the range $0<\lambda<2$. We need to subdivide the analysis into three cases: $0<\lambda<1$, $\lambda=1$, and $1<\lambda<2$. This is due to the different orders of growth of the sequence $(H_{\lambda}(\alpha_{N,\lambda})-N^{2} I_{\lambda}(\sigma_{1}))$. 

Following the strategy employed in \cite{LopWag}, we introduce the following definitions. 

\begin{definition}
For a fixed $p\in\mathbb{N}$, let $\Theta_{p}$ denote the set of all vectors $\vec{\theta}=(\theta_{1},\ldots,\theta_{p})$ for which there exists an infinite sequence $\mathcal{N}$ of integers $N=2^{n_{1}}+\cdots+2^{n_{p}}$, $n_1>n_2>\cdots>n_{p}\geq 0$, satisfying
\begin{equation}\label{eq:asympcondThetap}
\lim_{N\in\mathcal{N}}\frac{2^{n_{i}}}{N}=\theta_{i},\qquad \mbox{for all}\,\,i=1,\ldots,p.
\end{equation}
\end{definition}

Note that $\Theta_{p}\subset [0,1]^{p}$. If $(\theta_{1},\ldots,\theta_{p})\in\Theta_{p}$, then 
\[
\sum_{i=1}^{p}\theta_{i}=1
\]
and $\theta_{i}\geq 2\,\theta_{i+1}$ for each $1\leq i\leq p-1$.
In particular, $\theta_{1}>0$, and if $\theta_{i}=0$ for some $2\leq i\leq p$, then $\theta_{j}=0$ for all $i\leq j\leq p$.

The asymptotic condition \eqref{eq:asympcondThetap} is the defining property of $\Theta_{p}$ that we will apply, but this set can be defined alternatively as follows. $\Theta_{p}$ consists of all vectors $\vec{\theta}=(\theta_{1},\ldots,\theta_{p})$ that can be written in the form
\[
\vec{\theta}=\left(\frac{2^{t_{1}}}{M},\frac{2^{t_{2}}}{M},\ldots,\frac{2^{t_{r-1}}}{M},\frac{1}{M},0,\ldots,0\right),
\]
where $M=2^{t_{1}}+\cdots+2^{t_{r-1}}+1$ is an odd integer with $t_{1}>t_{2}>\cdots>t_{r-1}>0$ and $1\leq r\leq p$. The number of zeros that appear in $\vec{\theta}$, if any, is $p-r$.

\begin{definition}
On $\Theta_{p}\times [0,1)$ we define the function
\begin{equation}\label{def:Gfunction}
G((\theta_{1},\ldots,\theta_{p});\lambda):=\sum_{k=1}^{p}\theta_{k}^{-\lambda}\Big(2(2^{-\lambda}-1)\left(\sum_{j=k+1}^{p}\theta_{j}\right)+\theta_{k}\Big).
\end{equation}
The inner sum is understood to be zero if $k=p$. Also, if $\theta_{k}=0$ for some $2\leq k\leq p$, then the expression $\theta_{k}^{-\lambda}\Big(2(2^{-\lambda}-1)\left(\sum_{j=k+1}^{p}\theta_{j}\right)+\theta_{k}\Big)$ is understood to be zero, i.e.,
\begin{equation}\label{eq:convention}
\theta_{k}=0\implies \theta_{k}^{-\lambda}\Big(2(2^{-\lambda}-1)\left(\sum_{j=k+1}^{p}\theta_{j}\right)+\theta_{k}\Big)=0.
\end{equation}
\end{definition}

Observe that it is natural to adopt the convention \eqref{eq:convention} because $0\leq \lambda<1$, which implies $\lim_{\theta\rightarrow 0^{+}}\theta^{1-\lambda}=0$.

Another important observation is that for any $\vec{\theta}\in\Theta_{p}$, we have $G(\vec{\theta};\lambda)>0$. Indeed, from the inequalities $1-\lambda>0$ and $\theta_{i}\geq 2\,\theta_{i+1}$, $1\leq i\leq p-1$, it easily follows that for each $k$ with $\theta_{k}>0$, we have
\[
2(2^{-\lambda}-1)\left(\sum_{j=k+1}^{p}\theta_{j}\right)+\theta_{k}>0.
\]
As a function of $\lambda$, $G(\vec{\theta};\lambda)$ is continuous, and $G(\vec{\theta};0)=\sum_{k=1}^{p}\theta_{k}=1$.

\begin{figure}[h]
\centering
\includegraphics[totalheight=2in,keepaspectratio]{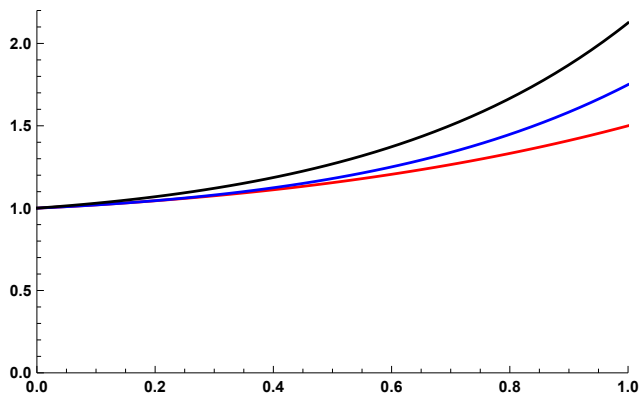}
\caption{The figure shows the graphs of the functions \eqref{def:Gfunction} associated with the vectors $\vec{\theta}=(2/3,1/3)$, $\vec{\theta}=(4/7,2/7,1/7)$, and $\vec{\theta}=(8/11, 2/11, 1/11)$, in increasing order.}
\label{fig:funcG}
\end{figure}

\begin{definition}
Let $0\leq \lambda<1$ be fixed. Using the function \eqref{def:Gfunction}, we define
\begin{align*}
\overline{g}_{p}(\lambda) & :=\sup_{\vec{\theta}\in\Theta_{p}}G(\vec{\theta};\lambda),\qquad p\in\mathbb{N},\\
\overline{g}(\lambda) & :=\sup_{p\in\mathbb{N}} \overline{g}_{p}(\lambda).
\end{align*} 
\end{definition}

Taking $(\theta_{1},\theta_{2})=(2/3,1/3)\in\Theta_{2}$, we have
\[
G((\theta_1,\theta_2);\lambda)=\frac{2^{1-2\lambda}+1}{3^{1-\lambda}}>1,\qquad 0<\lambda<1,
\]
see also Figure~\ref{fig:funcG}. Hence
\begin{equation}\label{eq:estoverg}
\overline{g}(\lambda)>1,\qquad 0<\lambda<1.
\end{equation}

A key identity we will use is the following, which the reader can easily check. If $N$ has the binary representation \eqref{eq:Nbinrep}, then
\begin{equation}\label{eq:binNsquare}
N^{2}=\sum_{k=1}^{p-1}\left(\sum_{j=k+1}^{p} 2^{n_{j}-n_{k}}\right) 2^{2(n_{k}+1)}+\sum_{k=1}^{p}\left(1-\sum_{j=k+1}^{p} 2^{n_{j}-n_{k}+1}\right) 2^{2 n_{k}}.
\end{equation}
Note the similarity with the expression \eqref{eq:energyalphaN}. We will also use the notation
\begin{equation}\label{def:Rcal}
\mathcal{R}_{\lambda}(N):=\frac{\mathcal{L}_{\lambda}(N)-N^{2} I_{\lambda}(\sigma_1)}{N^{1-\lambda}},\qquad N\geq 1.
\end{equation}
Then, according to \eqref{eq:asympexpLN}, we have
\begin{equation}\label{eq:asympRcal}
\lim_{N\rightarrow\infty}\mathcal{R}_{\lambda}(N)= (2\pi)^{\lambda}\, 2\zeta(-\lambda),\qquad 0<\lambda<2.
\end{equation}
It should also be noted that
\begin{equation}\label{ineq:negR}
\mathcal{R}_{\lambda}(N)<0,\qquad \mbox{for all}\,\,N\geq 1.
\end{equation}
This follows from \eqref{opteqspaced} and the second inequality in \eqref{ineqalphaomega}.

Given an integer $N\geq 1$, let $\tau(N)$ denote the number of ones in the binary digit representation of $N$, i.e., for an integer $N$ as in \eqref{eq:Nbinrep} we write $\tau(N)=p$.

The following result describes the second-order asymptotics of $H_{\lambda}(\alpha_{N,\lambda})$ in the range $0<\lambda<1$. 

\begin{theorem}\label{theo:casello}
Assume $0<\lambda<1$, let $(a_{n})_{n=0}^{\infty}\subset S^{1}$ be a greedy $\lambda$-energy sequence, and $\alpha_{N,\lambda}$, $N\geq 2$, be the configuration \eqref{def:alphaNlambda}. The sequence
\begin{equation}\label{eq:seqsecorderllo}
\left(\frac{H_{\lambda}(\alpha_{N,\lambda})-N^2\,I_{\lambda}(\sigma_{1})}{N^{1-\lambda}}\right)_{N=2}^{\infty}
\end{equation}
is bounded and divergent, and we have
\begin{align}
\limsup_{N\rightarrow\infty}\frac{H_{\lambda}(\alpha_{N,\lambda})-N^{2} I_{\lambda}(\sigma_{1})}{N^{1-\lambda}} & = (2\pi)^{\lambda}\,2\zeta(-\lambda),\label{eq:energysecord:1}\\
\liminf_{N\rightarrow\infty}\frac{H_{\lambda}(\alpha_{N,\lambda})-N^{2} I_{\lambda}(\sigma_{1})}{N^{1-\lambda}} & =\overline{g}(\lambda)\,(2\pi)^{\lambda}\,2\zeta(-\lambda).\label{eq:energysecord:2}
\end{align}
For any $p\geq 1$ and $\vec{\theta}\in\Theta_{p}$, the value $G(\vec{\theta};\lambda)\,(2\pi)^{\lambda}\,2\zeta(-\lambda)$ is a limit point\footnotemark[4]\footnotetext[4]{By limit point of the sequence we mean that there is a subsequence of \eqref{eq:seqsecorderllo} that converges to the indicated value.} of the sequence \eqref{eq:seqsecorderllo}.
\end{theorem}
\begin{proof}
Let $N\geq 2$ be an integer with binary representation \eqref{eq:Nbinrep}. Applying \eqref{eq:energyalphaN}, \eqref{eq:binNsquare}, and \eqref{def:Rcal}, we get
\begin{align}
\frac{H_{\lambda}(\alpha_{N,\lambda})-N^2 I_{\lambda}(\sigma_{1})}{N^{1-\lambda}} & =
\sum_{k=1}^{p-1}\left(\sum_{j=k+1}^{p} 2^{n_{j}-n_{k}}\right)\frac{\mathcal{L}_{\lambda}(2^{n_{k}+1})-2^{2(n_{k}+1)} I_{\lambda}(\sigma_{1})}{N^{1-\lambda}}\notag\\
& +\sum_{k=1}^{p}\left(1-\sum_{j=k+1}^{p} 2^{n_{j}-n_{k}+1}\right)\frac{\mathcal{L}_{\lambda}(2^{n_{k}})-2^{2 n_{k}} I_{\lambda}(\sigma_{1})}{N^{1-\lambda}}\notag\\
& =\sum_{k=1}^{p-1}\left(\sum_{j=k+1}^{p} 2^{n_{j}-n_{k}}\right)\left(\frac{2^{n_{k}+1}}{N}\right)^{1-\lambda}\mathcal{R}_{\lambda}(2^{n_k+1})\notag\\
& +\sum_{k=1}^{p}\left(1-\sum_{j=k+1}^{p} 2^{n_{j}-n_{k}+1}\right) \left(\frac{2^{n_{k}}}{N}\right)^{1-\lambda} \mathcal{R}_{\lambda}(2^{n_{k}}).\label{eq:energysecordbinrep}
\end{align}

First we show that the sequence \eqref{eq:seqsecorderllo} is bounded. The sequence \eqref{def:Rcal} converges, so there exists $C>0$ such that $|\mathcal{R}_{\lambda}(N)|<C$ for all $N\geq 1$. Therefore, applying the triangle inequality in \eqref{eq:energysecordbinrep},
\begin{gather*}
\left|\frac{H_{\lambda}(\alpha_{N,\lambda})-N^2 I_{\lambda}(\sigma_{1})}{N^{1-\lambda}}\right|\\
\leq C \left(\sum_{k=1}^{p-1}\left(\sum_{j=k+1}^{p} 2^{n_{j}-n_{k}}\right)\left(\frac{2^{n_{k}+1}}{N}\right)^{1-\lambda}+\sum_{k=1}^{p}\left(1+\sum_{j=k+1}^{p} 2^{n_{j}-n_{k}+1}\right) \left(\frac{2^{n_{k}}}{N}\right)^{1-\lambda}\right).
\end{gather*}
Now, we have $n_{j}-n_{k}\leq -(j-k)$ for every $k+1\leq j\leq p$, hence
\begin{equation}\label{eq:boundgeom}
\sum_{j=k+1}^{p} 2^{n_{j}-n_{k}}\leq 2^{-1}+2^{-2}+\cdots+2^{-(p-k)}<1.
\end{equation}
Similarly,
\begin{equation}\label{eq:cleverest}
\left(\frac{2^{n_{k}}}{N}\right)^{1-\lambda}\leq \left(\frac{2^{n_{k}}}{2^{n_{1}}}\right)^{1-\lambda}\leq 2^{-(1-\lambda)(k-1)}.
\end{equation}
So we have the estimate
\[
\left|\frac{H_{\lambda}(\alpha_{N,\lambda})-N^2 I_{\lambda}(\sigma_{1})}{N^{1-\lambda}}\right|\leq C\, (2^{1-\lambda}+3)\sum_{k=1}^{p}2^{-(1-\lambda)(k-1)}<C\,\frac{2^{1-\lambda}+3}{1-2^{\lambda-1}}<\infty.
\]
This justifies the boundedness of \eqref{eq:seqsecorderllo}.

For every $N\geq 2$, we have $H_{\lambda}(\alpha_{N,\lambda})\leq \mathcal{L}_{\lambda}(N)$ because $\mathcal{L}_{\lambda}(N)$ is the $\lambda$-energy of a configuration of $N$ equally spaced points on $S^{1}$, which is the largest among all $N$-point configurations on the circle. Since $\alpha_{2^{n},\lambda}$ is the configuration of the $2^{n}$-th roots of unity, we have $H_{\lambda}(\alpha_{N,\lambda})=\mathcal{L}_{\lambda}(N)$ for each $N=2^{n}$, hence
\[
\limsup_{N\rightarrow\infty}\frac{H_{\lambda}(\alpha_{N,\lambda})-N^2\,I_{\lambda}(\sigma_{1})}{N^{1-\lambda}}=\lim_{n\rightarrow\infty}\mathcal{R}_{\lambda}(2^{n})=(2\pi)^{\lambda}\,2\zeta(-\lambda).
\]
This is \eqref{eq:energysecord:1}.

Now we turn to the proof of \eqref{eq:energysecord:2}. Fix $p\geq 1$, and let $\vec{\theta}=(\theta_{1},\ldots,\theta_{p})\in\Theta_{p}$. By definition, there exists an infinite sequence $\mathcal{N}\subset \mathbb{N}$ such that each $N\in\mathcal{N}$ has a binary representation \eqref{eq:Nbinrep} of length $p$, and \eqref{eq:asympcondThetap} holds. Fix $1\leq k\leq p$, and assume first that the corresponding $\theta_{k}$ is not zero, i.e.,
\[
\lim_{N\in\mathcal{N}}\frac{2^{n_{k}}}{N}=\theta_{k}>0.
\]
In this case, we clearly have $2^{n_{k}}\rightarrow\infty$ as $N\in\mathcal{N}$ approaches infinity. Therefore, in virtue of \eqref{eq:asympcondThetap} and \eqref{eq:asympRcal}, we have
\begin{gather}
\lim_{N\in\mathcal{N}}\mathcal{R}_{\lambda}(2^{n_{k}})=\lim_{N\in\mathcal{N}}\mathcal{R}_{\lambda}(2^{n_{k}+1})=(2\pi)^{\lambda}\,2\zeta(-\lambda)\notag\\
\lim_{N\in\mathcal{N}}\left(\frac{2^{n_{k}+1}}{N}\right)^{1-\lambda}=(2\theta_{k})^{1-\lambda}\qquad \lim_{N\in\mathcal{N}}\left(\frac{2^{n_{k}}}{N}\right)^{1-\lambda}=\theta_{k}^{1-\lambda}\label{eq:plentylimits}\\
\lim_{N\in\mathcal{N}} 2^{n_{j}-n_{k}}=\frac{\theta_{j}}{\theta_{k}}\qquad \mbox{for each}\,\,k+1\leq j\leq p.\notag
\end{gather} 
Assume now that the index $k$ is such that
\[
\lim_{N\in\mathcal{N}}\frac{2^{n_{k}}}{N}=\theta_{k}=0.
\]
In this case, since $0\leq \lambda<1$, the sequence $(\mathcal{R}_{\lambda}(N))$ is bounded, and we have \eqref{eq:boundgeom}, it follows that
\begin{equation}\label{eq:plentylimits:2}
\begin{aligned}
\lim_{N\in\mathcal{N}}\left(\sum_{j=k+1}^{p} 2^{n_{j}-n_{k}}\right)\left(\frac{2^{n_{k}+1}}{N}\right)^{1-\lambda}\mathcal{R}_{\lambda}(2^{n_{k}+1}) & =0,\\
\lim_{N\in\mathcal{N}}\left(1-\sum_{j=k+1}^{p} 2^{n_{j}-n_{k}+1}\right)\left(\frac{2^{n_{k}}}{N}\right)^{1-\lambda}\mathcal{R}_{\lambda}(2^{n_{k}}) & =0.
\end{aligned}
\end{equation}
Suppose that $1\leq\widehat{\kappa}\leq p$ is the last component of $\vec{\theta}=(\theta_{1},\ldots,\theta_{p})$ that is non-zero, i.e.,
\begin{align*}
& \theta_{i}>0 \quad\mbox{for}\,\,1\leq i\leq \widehat{\kappa},\\
& \theta_{i}=0 \quad\mbox{for}\,\,\widehat{\kappa}+1\leq i\leq p.
\end{align*}
Then, applying \eqref{eq:energysecordbinrep}, \eqref{eq:plentylimits}, and \eqref{eq:plentylimits:2}, we obtain
\begin{align*}
\lim_{N\in\mathcal{N}}\frac{H_{\lambda}(\alpha_{N,\lambda})-N^{2}\,I_{\lambda}(\sigma_{1})}{N^{1-\lambda}}
& =\left(\sum_{k=1}^{\widehat{\kappa}}\theta_{k}^{-\lambda}(2(2^{-\lambda}-1)\left(\sum_{j=k+1}^{\widehat{\kappa}}\theta_{j}\right)+\theta_{k})\right) (2\pi)^{\lambda}\,2\zeta(-\lambda)\\
& =G((\theta_{1},\ldots,\theta_{p});\lambda)\,(2\pi)^{\lambda}\,2\zeta(-\lambda),
\end{align*}
where in the second equality we have used \eqref{eq:convention}. This shows that $G((\theta_{1},\ldots,\theta_{p});\lambda)\,(2\pi)^{\lambda}\,2\zeta(-\lambda)$ is a limit point of the sequence \eqref{eq:seqsecorderllo}. Since $p$ and $\vec{\theta}=(\theta_{1},\ldots,\theta_{p})$ were arbitrary, this implies 
\[
\liminf_{N\rightarrow\infty}\frac{H_{\lambda}(\alpha_{N,\lambda})-N^{2}\,I_{\lambda}(\sigma_{1})}{N^{1-\lambda}}\leq \overline{g}(\lambda)\,(2\pi)^{\lambda}\,2\zeta(-\lambda),
\] 
recall that $\zeta(-\lambda)<0$ for $\lambda\in(0,1)$.

To finish, we need to justify the inequality
\[
\liminf_{N\rightarrow\infty}\frac{H_{\lambda}(\alpha_{N,\lambda})-N^{2}\,I_{\lambda}(\sigma_{1})}{N^{1-\lambda}}\geq \overline{g}(\lambda)\,(2\pi)^{\lambda}\,2\zeta(-\lambda).
\]
\emph{Let} $\mathcal{N}\subset\mathbb{N}$ \emph{be an infinite sequence for which the sequence} $\left(\frac{H_{\lambda}(\alpha_{N,\lambda})-N^2\,I_{\lambda}(\sigma_{1})}{N^{1-\lambda}}\right)_{N\in\mathcal{N}}$ \emph{converges}, and let us show that
\begin{equation}\label{eq:ineqtbp}
\lim_{N\in\mathcal{N}}\frac{H_{\lambda}(\alpha_{N,\lambda})-N^2\,I_{\lambda}(\sigma_{1})}{N^{1-\lambda}}\geq 
\overline{g}(\lambda)\,(2\pi)^{\lambda}\,2\zeta(-\lambda).
\end{equation}

First assume that there exists $p\geq 1$ such that infinitely many $N\in\mathcal{N}$ satisfy $\tau(N)=p$. Then, passing to a subsequence $\widetilde{\mathcal{N}}$ of $\mathcal{N}$ if necessary, we may assume that the integers $N=2^{n_{1}}+\cdots+2^{n_{p}}\in\widetilde{\mathcal{N}}$ satisfy
\[
\lim_{N\in\widetilde{\mathcal{N}}} \frac{2^{n_{i}}}{N}=\theta_{i},\qquad \mbox{for all}\,\,i=1,\ldots,p.
\]
In this situation, we have already shown that we get
\[
\lim_{N\in\mathcal{N}}\frac{H_{\lambda}(\alpha_{N,\lambda})-N^2\,I_{\lambda}(\sigma_{1})}{N^{1-\lambda}}=
\lim_{N\in\widetilde{\mathcal{N}}}\frac{H_{\lambda}(\alpha_{N,\lambda})-N^2\,I_{\lambda}(\sigma_{1})}{N^{1-\lambda}}=G((\theta_{1},\ldots,\theta_{p});\lambda) (2\pi)^{\lambda}\,2\zeta(-\lambda),
\]
and \eqref{eq:ineqtbp} holds.

Assume now that such an integer $p$ does not exist. This means that we have $\tau(N)\rightarrow\infty$ as $N\rightarrow\infty$ along the sequence $\mathcal{N}$. For $N=2^{n_{1}}+\cdots+2^{n_{\tau(N)}}\in\mathcal{N}$, $n_{1}>n_{2}>\cdots>n_{\tau(N)}\geq 0$, we rewrite \eqref{eq:energysecordbinrep}:
\begin{gather}
\frac{H_{\lambda}(\alpha_{N,\lambda})-N^2\,I_{\lambda}(\sigma_{1})}{N^{1-\lambda}}\notag\\
=\sum_{k=1}^{\tau(N)}\left(\frac{2^{n_{k}}}{N}\right)^{1-\lambda}\left(2^{1-\lambda}\,\mathcal{R}_{\lambda}(2^{n_{k}+1})\sum_{j=k+1}^{\tau(N)}2^{n_{j}-n_{k}}+\mathcal{R}_{\lambda}(2^{n_{k}})\left(1-\sum_{j=k+1}^{\tau(N)} 2^{n_{j}-n_{k}+1}\right)\right).\label{eq:altexpsecord}
\end{gather}
The idea is now to truncate adequately the sums that appear in \eqref{eq:altexpsecord}. Let us define
\begin{align}
\eta_{N,k} & :=2^{1-\lambda}\,\mathcal{R}_{\lambda}(2^{n_{k}+1})\sum_{j=k+1}^{\tau(N)}2^{n_{j}-n_{k}}+\mathcal{R}_{\lambda}(2^{n_{k}})\left(1-\sum_{j=k+1}^{\tau(N)} 2^{n_{j}-n_{k}+1}\right)\notag\\
& =\mathcal{R}_{\lambda}(2^{n_{k}})+(2^{-\lambda}\mathcal{R}_{\lambda}(2^{n_{k}+1})-\mathcal{R}_{\lambda}(2^{n_{k}}))\sum_{j=k+1}^{\tau(N)} 2^{n_{j}-n_{k}+1}.\label{def:etaNk} 
\end{align}
The boundedness of the sequence $(\mathcal{R}_{\lambda}(N))$ and \eqref{eq:boundgeom} imply that there exists an absolute constant $C_{1}>0$ such that
\begin{equation}\label{eq:boundetaNk}
|\eta_{N,k}|\leq C_{1},\qquad \mbox{for all}\,\,N\in\mathcal{N}\,\,\mbox{and}\,\,1\leq k\leq \tau(N).
\end{equation}
Let $0<\epsilon<1$ be fixed. Let $M=M(\epsilon,\lambda)\geq 1$ be large enough so that
\begin{equation}\label{eq:truncconvseries}
\sum_{k=M}^{\infty}2^{-(1-\lambda)k}<\epsilon.
\end{equation}
The following choice suffices:
\begin{equation}\label{def:constM}
M:=\left\lceil-\frac{\log (\epsilon(2^{1-\lambda}-1))}{(1-\lambda)\log 2}\right\rceil+2,
\end{equation}
where $\lceil\cdot\rceil$ is the ceiling function.

We can write
\begin{equation}\label{eq:breaksecord}
\frac{H_{\lambda}(\alpha_{N,\lambda})-N^{2} I_{\lambda}(\sigma_{1})}{N^{1-\lambda}}=S_{N,M,1}+S_{N,M,2},
\end{equation}
where
\[
S_{N,M,1}:=\sum_{k=1}^{M}\left(\frac{2^{n_{k}}}{N}\right)^{1-\lambda} \eta_{N,k},\qquad S_{N,M,2}:=\sum_{k=M+1}^{\tau(N)}\left(\frac{2^{n_{k}}}{N}\right)^{1-\lambda} \eta_{N,k}.
\]
Then, from \eqref{eq:cleverest}, \eqref{eq:boundetaNk}, and \eqref{eq:truncconvseries}, we deduce
\begin{equation}\label{eq:boundSNM2}
|S_{N,M,2}|\leq C_{1}\epsilon,\qquad\mbox{for all}\,\,N\in\mathcal{N}.
\end{equation}
Now we concentrate our analysis on $S_{N,M,1}$. So in the rest of the proof, $k$ is an index in the range $1\leq k\leq M$.

We split the expression of $\eta_{N,k}$ in \eqref{def:etaNk} as follows:
\begin{equation}\label{eq:splitetaNkintwo}
\eta_{N,k}=\eta_{N,k,1}+\eta_{N,k,2},
\end{equation}
where
\begin{align*}
\eta_{N,k,1} & :=\mathcal{R_{\lambda}}(2^{n_{k}})+(2^{-\lambda}\,\mathcal{R}_{\lambda}(2^{n_{k}+1})-\mathcal{R}_{\lambda}(2^{n_{k}}))\sum_{j=k+1}^{2M} 2^{n_{j}-n_{k}+1},\\
\eta_{N,k,2} & :=(2^{-\lambda}\,\mathcal{R}_{\lambda}(2^{n_{k}+1})-\mathcal{R}_{\lambda}(2^{n_{k}}))\sum_{j=2M+1}^{\tau(N)} 2^{n_{j}-n_{k}+1}.
\end{align*}
Recall that $n_{j}-n_{k}\leq -(j-k)$ for all $j\geq k+1$, hence
if $1\leq k\leq M$ and $2M+1\leq j$, then
\[
\sum_{j=2M+1}^{\tau(N)} 2^{n_{j}-n_{k}+1}\leq \sum_{j=2M+1}^{\tau(N)} 2^{-j+M+1}<\sum_{t=M}^{\infty}2^{-t}<\epsilon
\]
where we used \eqref{eq:truncconvseries}. So by the boundedness of the sequence $(\mathcal{R}_{\lambda}(N))$, we obtain that there exists an absolute constant $C_{2}>0$ such that
\begin{equation}\label{eq:boundetaNk2}
|\eta_{N,k,2}|<C_{2}\epsilon,\qquad \mbox{for all}\,\,N\in\mathcal{N}\,\,\mbox{and}\,\,1\leq k\leq M.
\end{equation}
Now, by definition of $S_{N,M,1}$ and \eqref{eq:splitetaNkintwo}, we have
\[
S_{N,M,1}=\sum_{k=1}^{M}\left(\frac{2^{n_{k}}}{N}\right)^{1-\lambda} \eta_{N,k,1}+\sum_{k=1}^{M}\left(\frac{2^{n_{k}}}{N}\right)^{1-\lambda}\eta_{N,k,2}.
\]
From the estimate \eqref{eq:boundetaNk2} we get
\begin{equation}\label{eq:boundSNM1}
\left|\sum_{k=1}^{M}\left(\frac{2^{n_{k}}}{N}\right)^{1-\lambda}\eta_{N,k,2}\right|\leq C_{2} M \epsilon.
\end{equation}
We apply \eqref{eq:breaksecord}, \eqref{eq:boundSNM2}, and \eqref{eq:boundSNM1}, and conclude that
\begin{equation}\label{eq:decompsecord}
\begin{gathered}
\frac{H_{\lambda}(\alpha_{N,\lambda})-N^{2}\,I_{\lambda}(\sigma_{1})}{N^{1-\lambda}}=\phi_{N}+\sum_{k=1}^{M}\left(\frac{2^{n_{k}}}{N}\right)^{1-\lambda} \eta_{N,k,1}\\
=\phi_{N}+\sum_{k=1}^{M}\left(\frac{2^{n_{k}}}{N}\right)^{1-\lambda}\left(\mathcal{R_{\lambda}}(2^{n_{k}})+(2^{-\lambda}\,\mathcal{R}_{\lambda}(2^{n_{k}+1})-\mathcal{R}_{\lambda}(2^{n_{k}}))\sum_{j=k+1}^{2M} 2^{n_{j}-n_{k}+1}\right)
\end{gathered}
\end{equation}
where
\begin{equation}\label{eq:estphiN}
|\phi_{N}|\leq C_{1}\epsilon+C_{2}M\epsilon,\qquad \mbox{for all}\,\,N\in\mathcal{N}.
\end{equation}
Since $M$ is a fixed finite constant, there exists a subsequence $\widetilde{\mathcal{N}}\subset\mathcal{N}$ such that
\begin{equation}\label{eq:subseqassump}
\lim_{N\in\widetilde{\mathcal{N}}}\frac{2^{n_{i}}}{N}=\theta_{i},\qquad \mbox{for all}\,\,1\leq i\leq 2M.
\end{equation}
If $\theta_{k}>0$ for some $1\leq k\leq M$, then $2^{n_{k}}\rightarrow\infty$ as $N\in\widetilde{\mathcal{N}}$ tends to infinity, hence it follows from \eqref{eq:asympRcal} and \eqref{eq:subseqassump} that for this $k$, 
\begin{equation}\label{eq:finallimit}
\begin{gathered}
\lim_{N\in\widetilde{\mathcal{N}}}\left(\frac{2^{n_{k}}}{N}\right)^{1-\lambda}\Big(\mathcal{R_{\lambda}}(2^{n_{k}})+(2^{-\lambda}\,\mathcal{R}_{\lambda}(2^{n_{k}+1})-\mathcal{R}_{\lambda}(2^{n_{k}}))\sum_{j=k+1}^{2M} 2^{n_{j}-n_{k}+1}\Big)\\
=\theta_{k}^{-\lambda}\Big(2(2^{-\lambda}-1)\left(\sum_{j=k+1}^{2M}\theta_{j}\right)+\theta_{k}\Big) (2\pi)^{\lambda}\,2\zeta(-\lambda).
\end{gathered}
\end{equation}
And if $\theta_{k}=0$ for some $1\leq k\leq M$, then as explained before (see \eqref{eq:plentylimits:2}), the limit in \eqref{eq:finallimit} is zero. In this case we can also express the limit value zero as the right-hand side of \eqref{eq:finallimit}, keeping in mind the convention \eqref{eq:convention}. 

From \eqref{eq:decompsecord} and \eqref{eq:finallimit} we obtain
\begin{align*}
\lim_{N\in\mathcal{N}}\frac{H_{\lambda}(\alpha_{N,\lambda})-N^{2}\,I_{\lambda}(\sigma_{1})}{N^{1-\lambda}} & =\lim_{N\in\widetilde{\mathcal{N}}}\frac{H_{\lambda}(\alpha_{N,\lambda})-N^{2}\,I_{\lambda}(\sigma_{1})}{N^{1-\lambda}} \\
& = \lim_{N\in\widetilde{\mathcal{N}}}\phi_{N}+\lim_{N\in\widetilde{\mathcal{N}}}\sum_{k=1}^{M}\left(\frac{2^{n_{k}}}{N}\right)^{1-\lambda} \eta_{N,k,1}\\
& =\lim_{N\in\widetilde{\mathcal{N}}}\phi_{N}+(2\pi)^{\lambda}\,2\zeta(-\lambda)\Big(\sum_{k=1}^{M}\theta_{k}^{-\lambda}\Big(2(2^{-\lambda}-1)\left(\sum_{j=k+1}^{2M}\theta_{j}\right)+\theta_{k}\Big)\Big)\\
& \geq \lim_{N\in\widetilde{\mathcal{N}}}\phi_{N}+(2\pi)^{\lambda}\,2\zeta(-\lambda)\Big(\sum_{k=1}^{M}\theta_{k}^{-\lambda}\Big(2(2^{-\lambda}-1)\left(\sum_{j=k+1}^{M}\theta_{j}\right)+\theta_{k}\Big)\Big)\\
& =\lim_{N\in\widetilde{\mathcal{N}}}\phi_{N}+(2\pi)^{\lambda}\,2\zeta(-\lambda)\,G((\theta_{1},\ldots,\theta_{M});\lambda)\\
& \geq \lim_{N\in\widetilde{\mathcal{N}}}\phi_{N}+\overline{g}(\lambda)\,(2\pi)^{\lambda}\,2\zeta(-\lambda)\\
& \geq -(C_{1}\epsilon+C_{2} M(\epsilon)\epsilon)+\overline{g}(\lambda)\,(2\pi)^{\lambda}\,2\zeta(-\lambda)
\end{align*}
where we used \eqref{eq:estphiN} in the last inequality. Letting $\epsilon\rightarrow 0$, we get $M(\epsilon)\epsilon\rightarrow 0$ (see \eqref{def:constM}), and \eqref{eq:ineqtbp} is proved. This finishes the proof of \eqref{eq:energysecord:2}.

The estimate \eqref{eq:estoverg} implies that the limit values in \eqref{eq:energysecord:1} and \eqref{eq:energysecord:2} are different, so the sequence \eqref{eq:seqsecorderllo} diverges.
\end{proof}

In Figure~\ref{plotsecorder2} we present plots of the normalized second-order sequences \eqref{eq:seqsecorderllo}, for different values of $\lambda\in(0,1)$.  

\begin{figure}[h]
\centering
\begin{subfigure}[b]{0.475\textwidth}
\centering
\includegraphics[width=\textwidth]{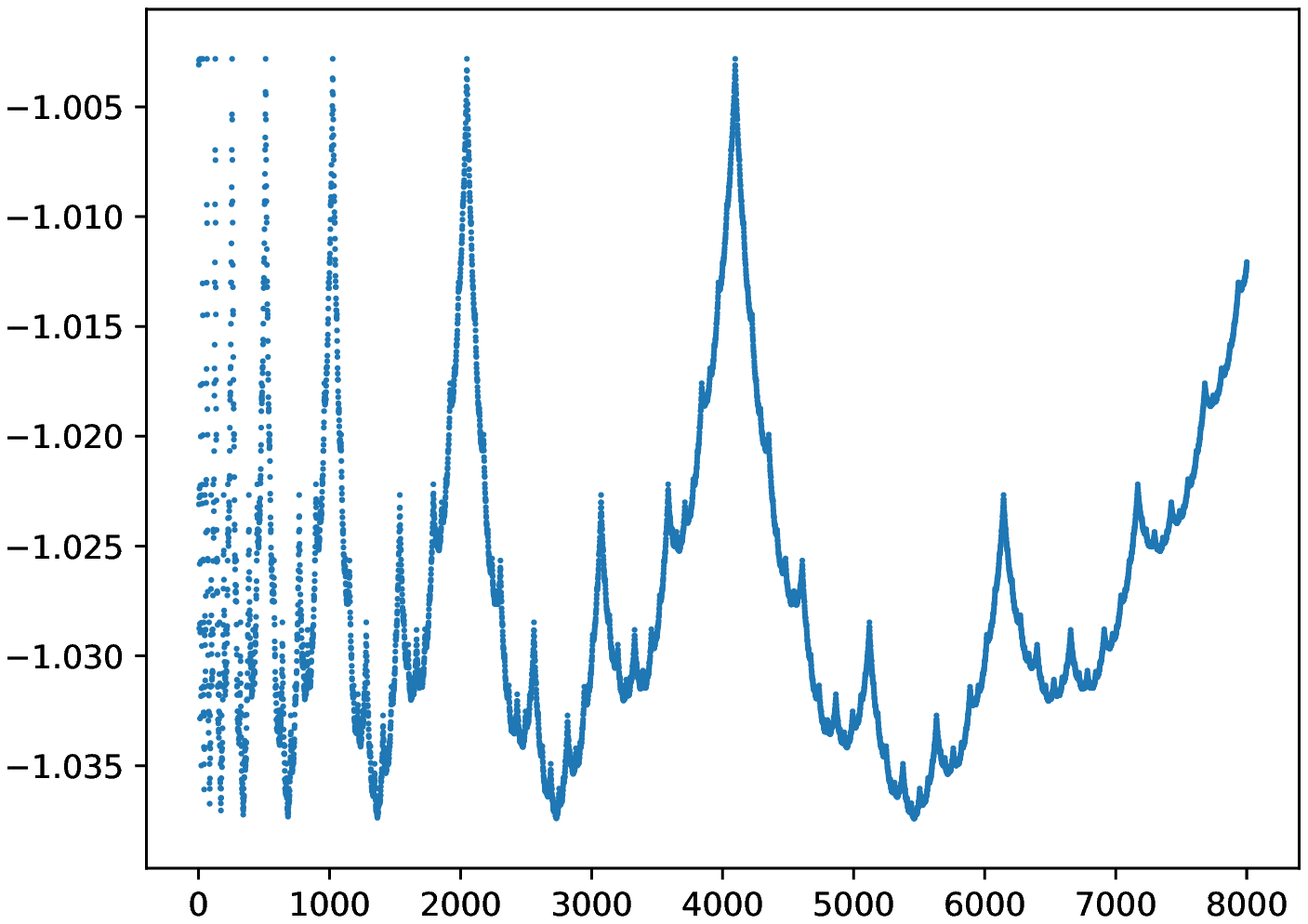}
\caption[]{{$\lambda = 0.1$}}
\end{subfigure}
\hfill
\begin{subfigure}[b]{0.475\textwidth}  
\centering 
\includegraphics[width=\textwidth]{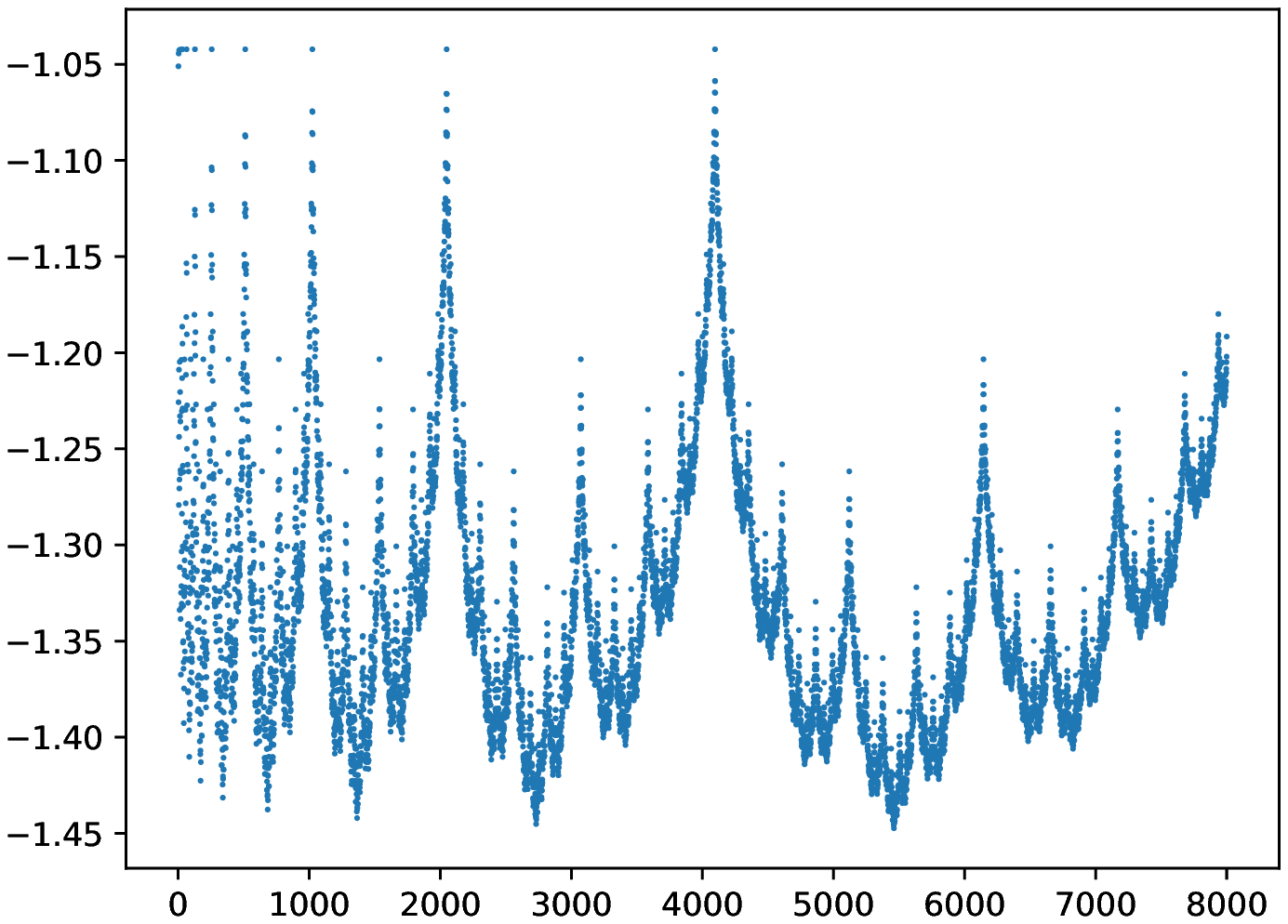}
\caption[]{{$\lambda = 0.5$}}
\end{subfigure}
\vskip\baselineskip
\begin{subfigure}[b]{0.475\textwidth}   
\centering 
\includegraphics[width=\textwidth]{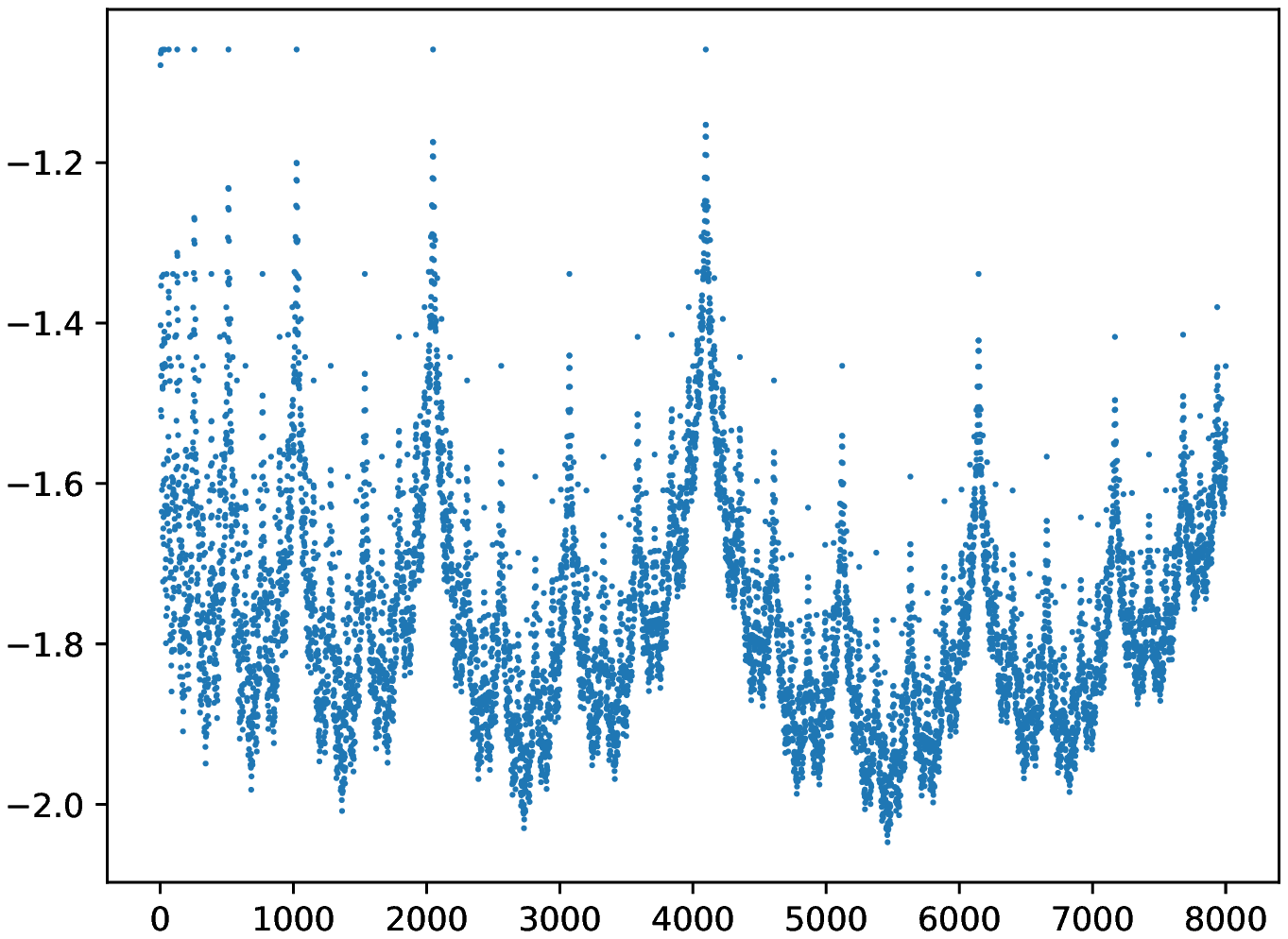}
\caption[]{{$\lambda = 0.7$}}
\end{subfigure}
\quad
\begin{subfigure}[b]{0.475\textwidth}   
\centering 
\includegraphics[width=\textwidth]{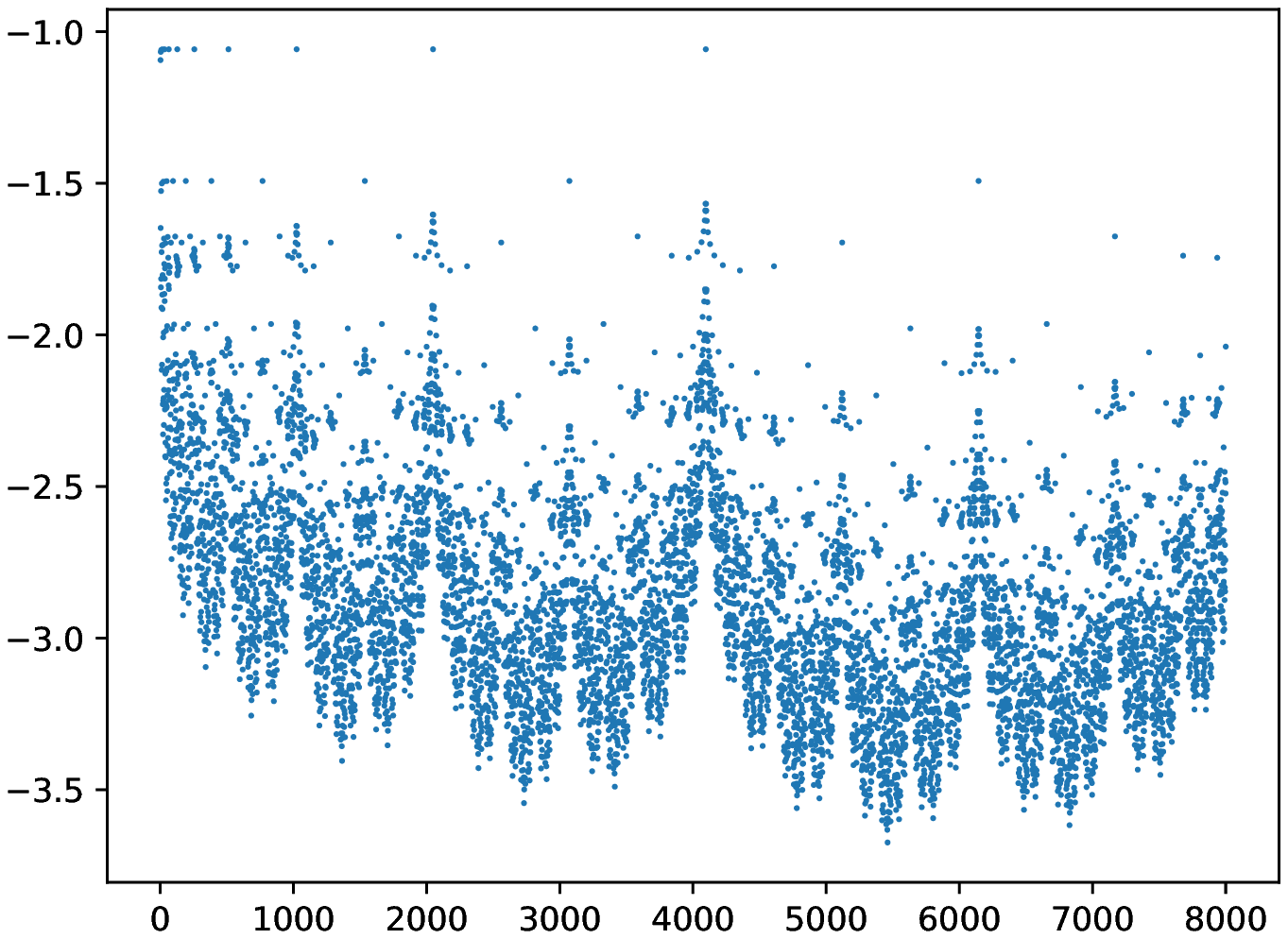}
\caption[]{{$\lambda = 0.9$}}
\end{subfigure}
\caption[]{{Plots of the normalized second-order energy \eqref{eq:seqsecorderllo} for $0<\lambda<1$ and $2\leq N\leq 8000$.}}
\label{plotsecorder2}
\end{figure}

The next result concerns the second-order asymptotics of $H_{\lambda}(\alpha_{N,\lambda})$ in the range $1<\lambda<2$.

\begin{theorem}\label{theo:caselgo}
Assume $1<\lambda<2$, let $(a_{n})_{n=0}^{\infty}\subset S^{1}$ be a greedy $\lambda$-energy sequence, and $\alpha_{N,\lambda}$, $N\geq 2$, be the configuration \eqref{def:alphaNlambda}. The sequence
\begin{equation}\label{eq:secordseqlgt1}
(H_{\lambda}(\alpha_{N,\lambda})-N^{2} I_{\lambda}(\sigma_{1}))_{N=2}^{\infty}
\end{equation}
is bounded and divergent. We have
\begin{align}
\limsup_{N\rightarrow\infty}\,(H_{\lambda}(\alpha_{N,\lambda})-N^2\,I_{\lambda}(\sigma_{1})) & =0 \label{eq:secordlgt1:1}\\
\liminf_{N\rightarrow\infty}\,(H_{\lambda}(\alpha_{N,\lambda})-N^2\,I_{\lambda}(\sigma_{1})) & \leq s_{\lambda} \label{eq:secordlgt1:2}
\end{align}
where\footnotemark[4]\footnotetext[4]{Our numerical computations of the terms in the sequence \eqref{eq:secordseqlgt1} seem to indicate that the constant $s_{\lambda}$ is in fact the liminf of the sequence \eqref{eq:secordseqlgt1}.}
\begin{equation}\label{def:slambda}
s_{\lambda}:=\frac{1}{3}\sum_{k=0}^{\infty}
\left(1+\frac{(-1)^{k}}{2^{k-1}}\right)(\mathcal{L}_{\lambda}(2^{k})-2^{2k} I_{\lambda}(\sigma_{1})).
\end{equation}
We also have
\begin{equation}\label{eq:liminfinf}
\liminf_{N\rightarrow\infty}\,(H_{\lambda}(\alpha_{N,\lambda})-N^2\,I_{\lambda}(\sigma_{1}))=\inf_{N\geq 2}\,\{H_{\lambda}(\alpha_{N,\lambda})-N^2\,I_{\lambda}(\sigma_{1})\}.
\end{equation}
Each term in the sequence \eqref{eq:secordseqlgt1} is a limit point of the sequence itself.
\end{theorem}
\begin{proof}
Let $N\geq 2$ be as in \eqref{eq:Nbinrep}. Applying \eqref{eq:energyalphaN}, \eqref{eq:binNsquare}, and \eqref{def:Rcal}, we have
\begin{equation}\label{eq:diffsecord}
\begin{aligned}
H_{\lambda}(\alpha_{N,\lambda})-N^2 I_{\lambda}(\sigma_{1}) & =
\sum_{k=1}^{p-1}\Big(\sum_{j=k+1}^{p} 2^{n_{j}-n_{k}}\Big)\left(2^{n_{k}+1}\right)^{1-\lambda}\,\mathcal{R}_{\lambda}(2^{n_k+1})\\
& +\sum_{k=1}^{p}\Big(1-\sum_{j=k+1}^{p} 2^{n_{j}-n_{k}+1}\Big) \left(2^{n_{k}}\right)^{1-\lambda}\,\mathcal{R}_{\lambda}(2^{n_{k}}).
\end{aligned}
\end{equation}
The sequence $(\mathcal{R}_{\lambda}(N))$ is bounded, and let $C>0$ be a constant such that $|\mathcal{R}_{\lambda}(N)|<C$ for all $N$. Applying the triangle inequality in \eqref{eq:diffsecord}, we get
\begin{gather*}
|H_{\lambda}(\alpha_{N,\lambda})-N^2 I_{\lambda}(\sigma_{1})| \\ 
\leq C \left(\sum_{k=1}^{p-1}\Big(\sum_{j=k+1}^{p} 2^{n_{j}-n_{k}}\Big)\left(2^{n_{k}+1}\right)^{1-\lambda}+\sum_{k=1}^{p}\Big(1+\sum_{j=k+1}^{p} 2^{n_{j}-n_{k}+1}\Big) \left(2^{n_{k}}\right)^{1-\lambda}\right).
\end{gather*}
Using \eqref{eq:boundgeom} and the estimate $\sum_{k=1}^{p} (2^{n_{k}})^{1-\lambda}<\sum_{t=0}^{\infty}(2^{1-\lambda})^{t}=1/(1-2^{1-\lambda})$, we obtain the uniform bound
\[
|H_{\lambda}(\alpha_{N,\lambda})-N^2 I_{\lambda}(\sigma_{1})|\leq C\frac{2^{1-\lambda}+3}{1-2^{1-\lambda}},
\]
so the sequence \eqref{eq:secordseqlgt1} is bounded. 

We have, for all $N\geq 2$,
\[
H_{\lambda}(\alpha_{N,\lambda})-N^{2} I_{\lambda}(\sigma_{1})\leq \mathcal{L}_{\lambda}(N)-N^{2} I_{\lambda}(\sigma_{1})=N^{1-\lambda}\,\mathcal{R}_{\lambda}(N),
\]
with equality if $N=2^{n}$, $n\geq 1$. Since $\lim_{N\rightarrow\infty} N^{1-\lambda}\,\mathcal{R}_{\lambda}(N)=0$, we obtain \eqref{eq:secordlgt1:1}.

Let us justify the identity \eqref{eq:liminfinf} now. If we show that for each fixed $\widehat{N}=2^{n_{1}}+\cdots+2^{n_{p}}$, $n_{1}>\ldots>n_{p}$, we have 
\begin{equation}\label{ineqsecordNhat}
\liminf_{N\rightarrow\infty}\,(H_{\lambda}(\alpha_{N,\lambda})-N^2 I_{\lambda}(\sigma_{1}))\leq H_{\lambda}(\alpha_{\widehat{N},\lambda})-\widehat{N}^2 I_{\lambda}(\sigma_{1}),
\end{equation}
then \eqref{eq:liminfinf} will be justified. For such $\widehat{N}$ fixed, take $\widetilde{N}=2^{n_{0}}+\widehat{N}$, and let $n_{0}\rightarrow\infty$. Since $(2^{n_{0}})^{1-\lambda}\rightarrow 0$, it easily follows from \eqref{eq:diffsecord} that
\[
\lim_{n_{0}\rightarrow\infty}\,(H_{\lambda}(\alpha_{\widetilde{N},\lambda})-\widetilde{N}^2 I_{\lambda}(\sigma_{1}))=H_{\lambda}(\alpha_{\widehat{N},\lambda})-\widehat{N}^2 I_{\lambda}(\sigma_{1}),
\]
which justifies \eqref{ineqsecordNhat}, and shows that every term in the sequence \eqref{eq:secordseqlgt1} is a limit point of the sequence itself.

Now we justify \eqref{eq:secordlgt1:2}. Consider the subsequence
\begin{equation}\label{def:specsub}
N(p):=\sum_{k=0}^{p-1}2^{2k}=\frac{4^{p}-1}{3}.
\end{equation}
Applying \eqref{eq:diffsecord} for $N=N(p)$, we have 
\begin{align}\label{eq:spsb}
H_{\lambda}(\alpha_{N,\lambda})-N^2 I_{\lambda}(\sigma_{1}) & =
\sum_{k=1}^{p-1}\frac{1}{3}\Big(1-\frac{1}{4^{p-k}}\Big)\left(2^{2(p-k)+1}\right)^{1-\lambda}\,\mathcal{R}_{\lambda}(2^{2(p-k)+1})\notag\\
& +\sum_{k=1}^{p}\frac{1}{3}\Big(1+\frac{2}{4^{p-k}}\Big)\left(2^{2(p-k)}\right)^{1-\lambda}\,\mathcal{R}_{\lambda}(2^{2(p-k)})\notag\\
& =\frac{1}{3}\sum_{m=0}^{p-1}\Big(1-\frac{1}{4^{m}}\Big)
(\mathcal{L}_{\lambda}(2^{2m+1})-2^{4m+2} I_{\lambda}(\sigma_{1}))\notag\\
&+\frac{1}{3}\sum_{m=0}^{p-1}\Big(1+\frac{2}{4^{m}}\Big)
(\mathcal{L}_{\lambda}(2^{2m})-2^{4m} I_{\lambda}(\sigma_{1}))
\end{align}
where in the second equality we applied the substitution $m=p-k$ and \eqref{def:Rcal}. This sequence is bounded and monotonic (every term in the sums is $\leq 0$), so it converges to
\begin{align*}
\lim_{p\rightarrow\infty}\,(H_{\lambda}(\alpha_{N,\lambda})-N^2 I_{\lambda}(\sigma_{1})) & = \frac{1}{3}\sum_{m=0}^{\infty}\Big(1-\frac{1}{4^{m}}\Big)
(\mathcal{L}_{\lambda}(2^{2m+1})-2^{4m+2} I_{\lambda}(\sigma_{1}))\\
& +\frac{1}{3}\sum_{m=0}^{\infty}\Big(1+\frac{2}{4^{m}}\Big)
(\mathcal{L}_{\lambda}(2^{2m})-2^{4m} I_{\lambda}(\sigma_{1}))
\end{align*}
and this is exactly \eqref{def:slambda}. 

Each term in the series \eqref{def:slambda} is negative, except the second one, which is zero. The first term in the series is $-I_{\lambda}(\sigma_{1})$, so $s_{\lambda}<-I_{\lambda}(\sigma_{1})<0$. Hence, the sequence \eqref{eq:secordseqlgt1} is divergent. 
\end{proof}

In Figure~\ref{plotsecord4} we present some plots of sequences \eqref{eq:secordseqlgt1} in the range $1<\lambda<2$.

\begin{figure}[h]
\centering
\begin{subfigure}[b]{0.475\textwidth}
\centering
\includegraphics[width=\textwidth]{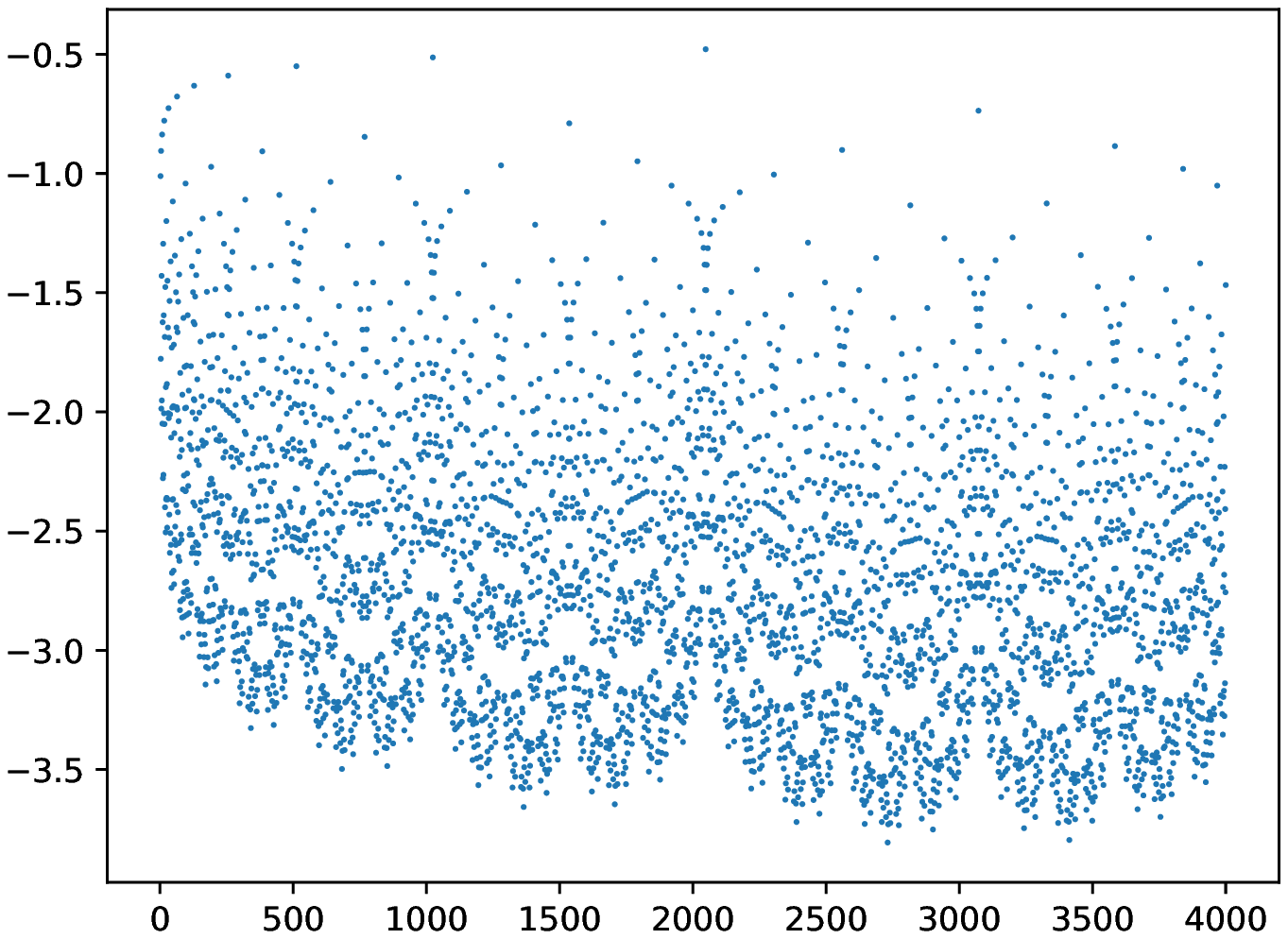}
\caption[]{{$\lambda = 1.1$}}    
\end{subfigure}
\hfill
\begin{subfigure}[b]{0.475\textwidth}  
\centering 
\includegraphics[width=\textwidth]{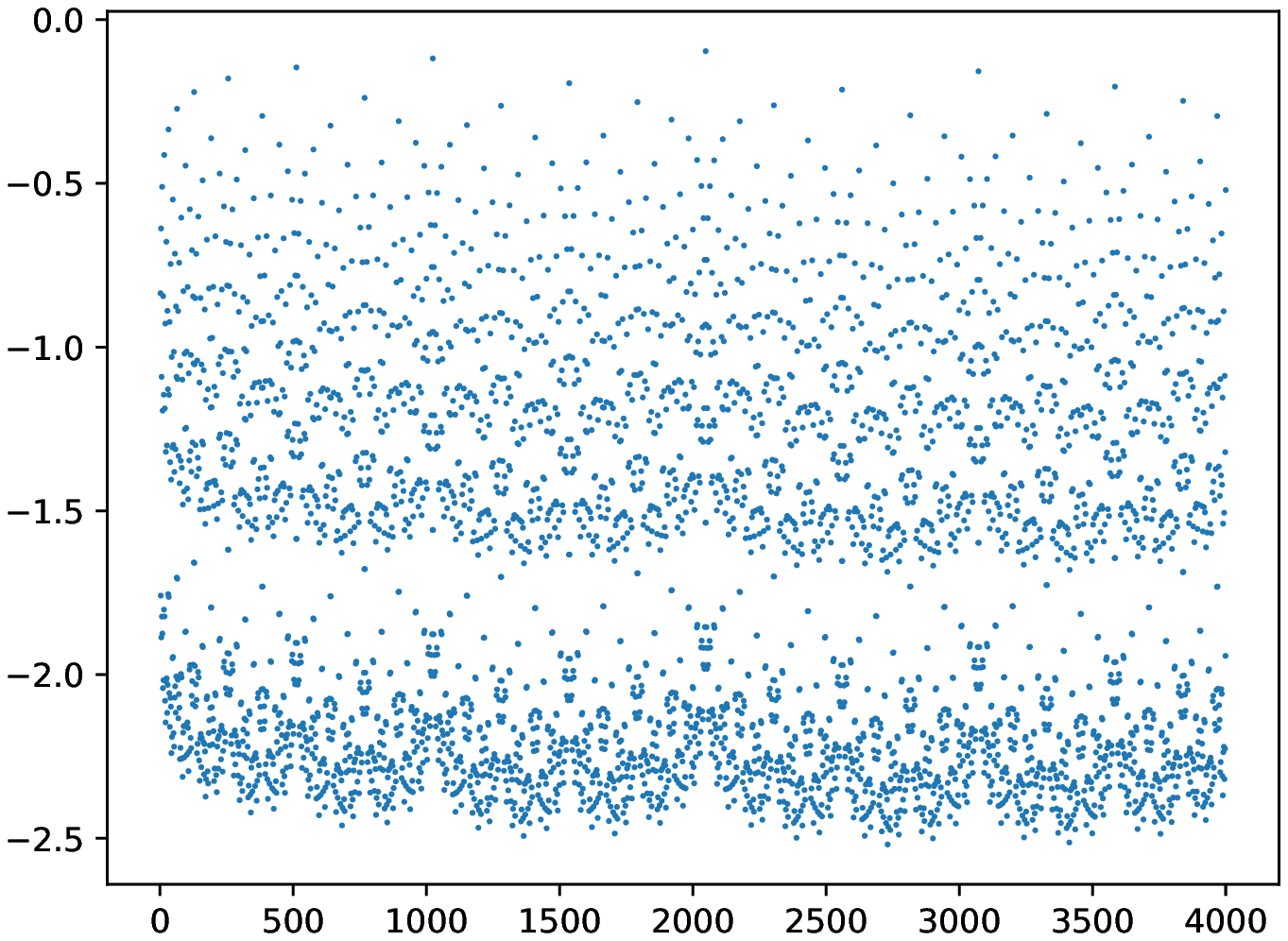}
\caption[]{{$\lambda = 1.3$}}
\end{subfigure}
\vskip\baselineskip
\begin{subfigure}[b]{0.475\textwidth}   
\centering 
\includegraphics[width=\textwidth]{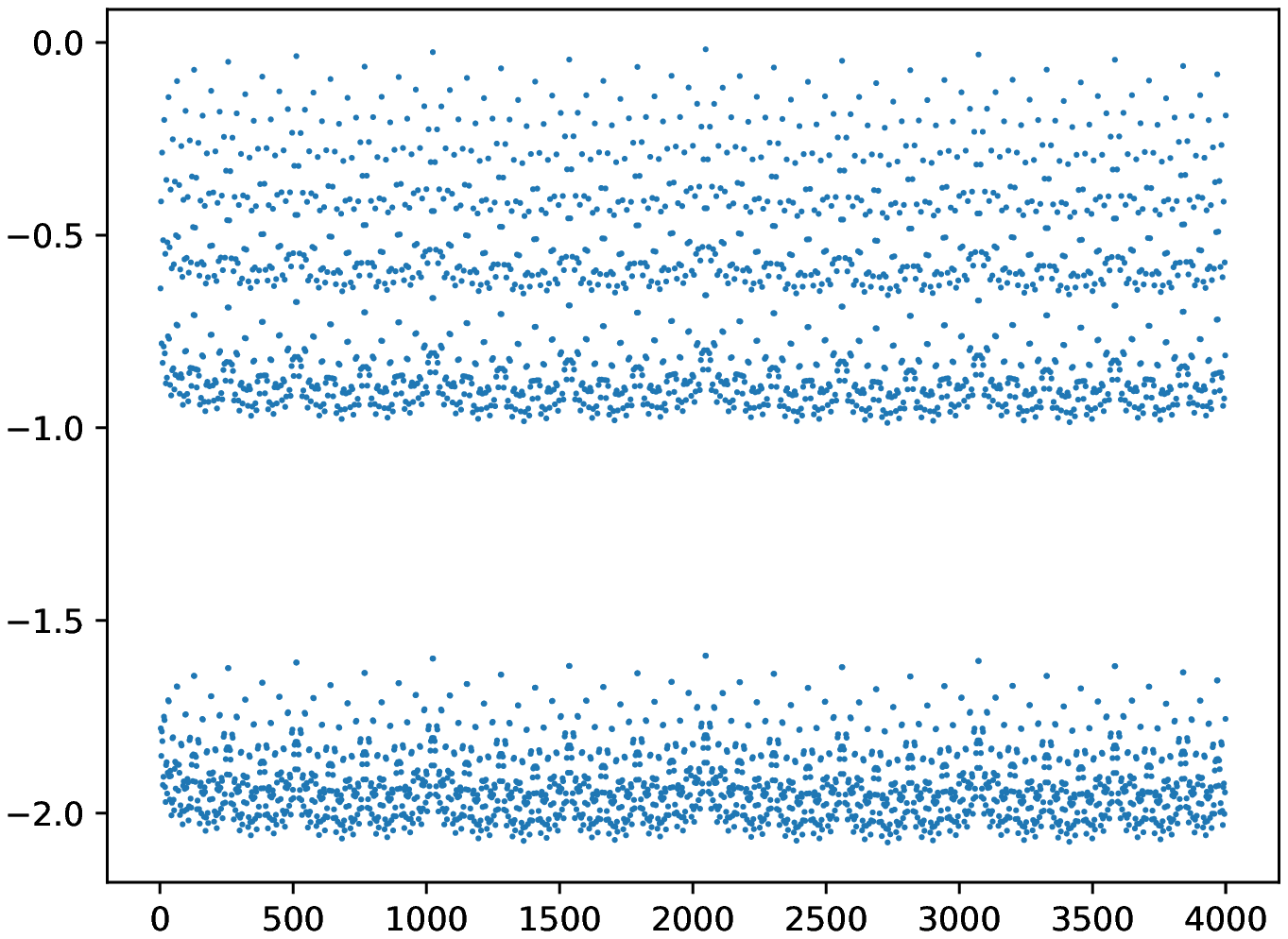}
\caption[]{{$\lambda = 1.5$}}    
\end{subfigure}
\quad
\begin{subfigure}[b]{0.475\textwidth}   
\centering 
\includegraphics[width=\textwidth]{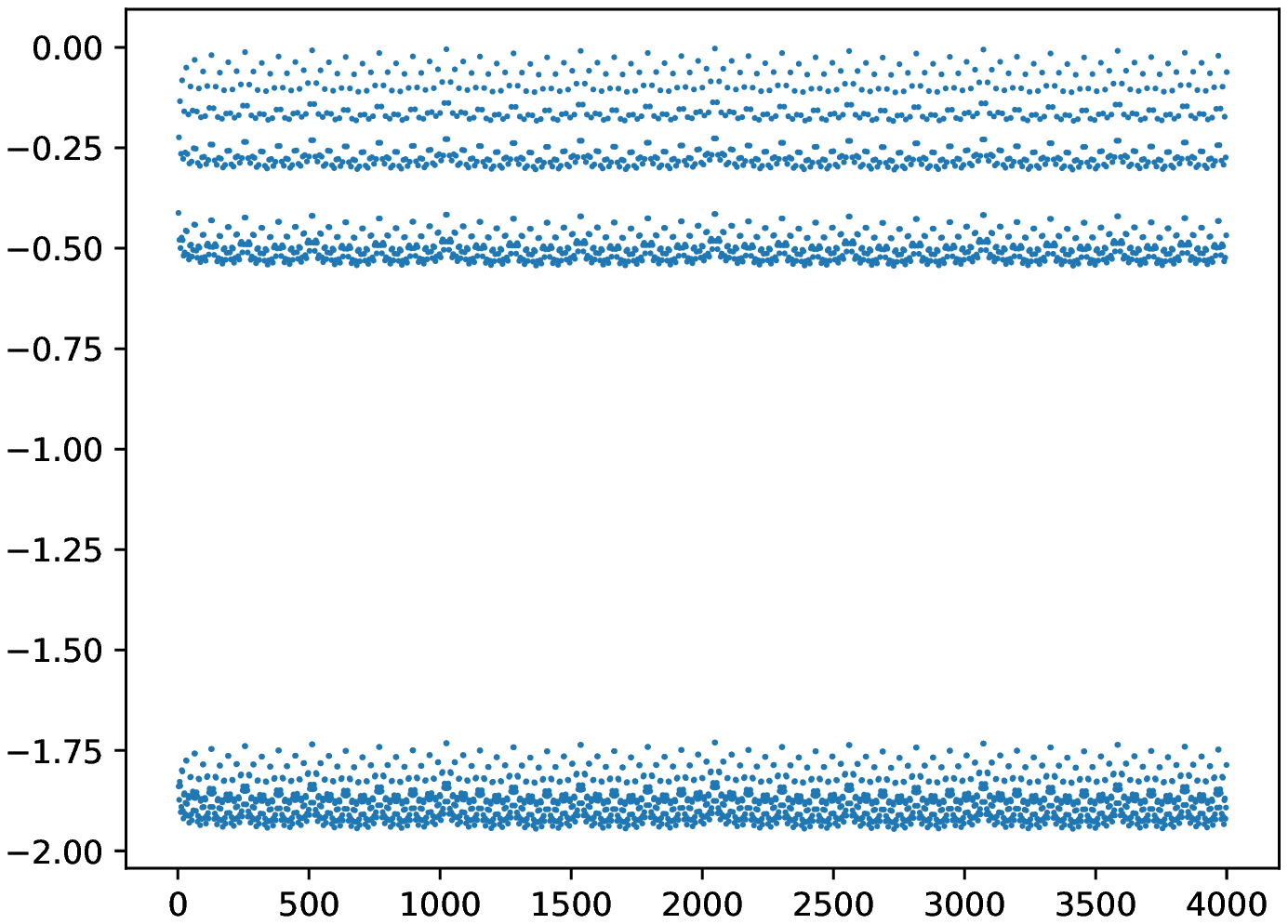}
\caption[]{{$\lambda = 1.7$}}    
\end{subfigure}
\caption[]{{Plots of sequences \eqref{eq:secordseqlgt1} for $1<\lambda<2$ and $2\leq N\leq 4000$.}}
\label{plotsecord4}
\end{figure}

Our next result concerns the second-order asymptotics of $H_{\lambda}(\alpha_{N,\lambda})$ in the case $\lambda=1$.

\begin{theorem}\label{theo:lambdaeq1}
Let $(a_{n})_{n=0}^{\infty}\subset S^{1}$ be a greedy $1$-energy sequence, and let $\alpha_{N,1}$, $N\geq 2$, be the configuration in \eqref{def:alphaNlambda}. The sequence
\begin{equation}\label{eq:secordexplamone}
\left(\frac{H_{1}(\alpha_{N,1})-N^{2} I_{1}(\sigma_{1})}{\log N}\right)_{N=2}^{\infty}
\end{equation}
is bounded and divergent. We have
\begin{align}
\limsup_{N\rightarrow\infty}\frac{H_{1}(\alpha_{N,1})-N^{2} I_{1}(\sigma_{1})}{\log N} & =0,\label{eq:limsupl1}\\
\liminf_{N\rightarrow\infty}\frac{H_{1}(\alpha_{N,1})-N^{2} I_{1}(\sigma_{1})}{\log N} & \leq -\frac{\pi}{9 \log 2}.\label{eq:liminfl1} 
\end{align}
\end{theorem}
\begin{proof}
For $N\geq 2$ as in \eqref{eq:Nbinrep}, we have
\begin{equation}\label{eq:expsecordleo}
\begin{aligned}
\frac{H_{1}(\alpha_{N,1})-N^2 I_{1}(\sigma_{1})}{\log N} & =
\sum_{k=1}^{p-1}\Big(\sum_{j=k+1}^{p} 2^{n_{j}-n_{k}}\Big)\,\frac{\mathcal{R}_{1}(2^{n_k+1})}{\log N}\\
& +\sum_{k=1}^{p}\Big(1-\sum_{j=k+1}^{p} 2^{n_{j}-n_{k}+1}\Big)\,\frac{\mathcal{R}_{1}(2^{n_{k}})}{\log N}.
\end{aligned}
\end{equation}
If $C$ is an upper bound for all $|\mathcal{R}_{1}(N)|$, applying \eqref{eq:boundgeom} and the triangle inequality we get
\begin{equation}\label{eq:acot}
\left|\frac{H_{1}(\alpha_{N,1})-N^2 I_{1}(\sigma_{1})}{\log N}\right|\leq \frac{4\,p \,C}{\log N}.
\end{equation}
It follows from \eqref{eq:Nbinrep} that $p\leq n_{1}+1$, and we have $\log N\geq \log (2^{n_{1}})=n_{1} \log 2$. These inequalities and \eqref{eq:acot} imply the boundedness of \eqref{eq:secordexplamone}.

We have $H_{1}(\alpha_{N,1})-N^{2} I_{1}(\sigma_{1})\leq \mathcal{R}_{1}(N)$ with equality if $N=2^{n}$, $n\geq 1$. Since we have $\lim_{N\rightarrow\infty} \mathcal{R}_{1}(N)/\log N=0$, we obtain \eqref{eq:limsupl1}.

The function $\zeta(s)$ satisfies $\zeta(-1)=-1/12$, hence according to \eqref{eq:asympRcal} we have
\begin{equation}\label{eq:asympR1}
\lim_{N\rightarrow\infty}\mathcal{R}_{1}(N)=-\frac{\pi}{3}.
\end{equation}
Consider the subsequence $N(p):=\frac{4^{p}-1}{3}$ defined in \eqref{def:specsub}. In virtue of \eqref{eq:spsb}, we have
\[
\frac{H_{1}(\alpha_{N(p),1})-N(p)^2 I_{1}(\sigma_{1})}{\log(N(p))}=\rho_{p,1}+\rho_{p,2}
\] 
where
\begin{align*}
\rho_{p,1} & :=\frac{1}{3}\,\frac{1}{\log(N(p))}\,\sum_{m=1}^{p-1}
\left(1-\frac{1}{4^{m}}\right)\,\mathcal{R}_{1}(2^{2m+1}),\\
\rho_{p,2} & :=\frac{1}{3}\,\frac{1}{\log(N(p))}\,\sum_{m=0}^{p-1}
\left(1+\frac{2}{4^{m}}\right)\,\mathcal{R}_{1}(2^{2m}).
\end{align*}

Let $\epsilon>0$ be fixed. By \eqref{eq:asympR1}, there exists $N_{\epsilon}\in\mathbb{N}$, which we assume to be odd, such that
\begin{equation}\label{eq:estR1}
\left|\mathcal{R}_{1}(2^{k})+\frac{\pi}{3}\right|<\epsilon \qquad\mbox{for all}\,\,k\geq N_{\epsilon}.
\end{equation}
Let $m_{\epsilon}:=(N_{\epsilon}-1)/2$. Then, we can write
\begin{align*}
\rho_{p,1} & =\frac{1}{3}\,\frac{1}{\log(N(p))}\left(\sum_{m=1}^{m_{\epsilon}-1}\left(1-\frac{1}{4^{m}}\right)\mathcal{R}_{1}(2^{2m+1})+\sum_{m=m_{\epsilon}}^{p-1}\left(1-\frac{1}{4^{m}}\right)\left(\mathcal{R}_{1}(2^{2m+1})+\frac{\pi}{3}\right)\right)\\
 & -\frac{\pi}{9}\frac{1}{\log(N(p))}\sum_{m=m_{\epsilon}}^{p-1}\left(1-\frac{1}{4^{m}}\right).
\end{align*}
Since $m_{\epsilon}$ is fixed, we have
\[
\lim_{p\rightarrow\infty}\frac{1}{\log(N(p))}\,\sum_{m=1}^{m_{\epsilon}-1}\left(1-\frac{1}{4^{m}}\right)\mathcal{R}_{1}(2^{2m+1})=0.
\]
Hence, applying \eqref{eq:estR1} we get
\begin{equation}\label{eq:estrhop1}
\left|\rho_{p,1}+\frac{\pi}{9}\frac{1}{\log(N(p))}\sum_{m=m_{\epsilon}}^{p-1}\left(1-\frac{1}{4^{m}}\right)\right|\leq o(1)+\frac{1}{3}\,\frac{\epsilon}{\log(N(p))}\,\sum_{m=m_{\epsilon}}^{p-1}\left(1-\frac{1}{4^{m}}\right).
\end{equation}
A simple calculation gives
\[
\lim_{p\rightarrow\infty}\frac{1}{\log(N(p))}\sum_{m=m_{\epsilon}}^{p-1}\left(1-\frac{1}{4^{m}}\right)=\frac{1}{\log 4}.
\]
So, letting $p\rightarrow\infty$ and then $\epsilon\rightarrow 0$ in \eqref{eq:estrhop1}, we get
\[
\lim_{p\rightarrow\infty} \rho_{p,1}=-\frac{\pi}{9 \log 4}.
\]
The same limit is valid for $\rho_{p,2}$, hence
\[
\liminf_{N\rightarrow\infty}\frac{H_{1}(\alpha_{N,1})-N^{2} I_{1}(\sigma_{1})}{\log N}\leq \lim_{p\rightarrow\infty}\frac{H_{1}(\alpha_{N(p),1})-N(p)^2 I_{1}(\sigma_{1})}{\log(N(p))}=-\frac{\pi}{9 \log 2},
\] 
which justifies \eqref{eq:liminfl1}.
\end{proof}

Figure~\ref{plotsecorder3} shows a plot of the sequence \eqref{eq:secordexplamone}. We remark that the value $-\pi/(9 \log 2)$ in the right-hand side of \eqref{eq:liminfl1} is the lowest limit value that we can get among subsequences $(N_{r}(p))_{p=1}^{\infty}$, where
\begin{equation}\label{def:subseqNrp}
N_{r}(p)=\sum_{k=0}^{p-1} 2^{rk}=\frac{2^{rp}-1}{2^{r}-1},\qquad r\in\mathbb{N},
\end{equation}
of which the sequence $(N(p))_{p=1}^{\infty}$ considered in the proof of Theorem~\ref{theo:lambdaeq1} is a particular case. Indeed, the reader can check that for the sequence \eqref{def:subseqNrp} we have 
\begin{equation}\label{eq:lpleo}
\lim_{p\rightarrow\infty}\frac{H_{1}(\alpha_{N_{r}(p),1})-N_{r}(p)^2 I_{1}(\sigma_{1})}{\log(N_{r}(p))}=-\frac{2^{r}-2}{r(2^{r}-1)}\,\frac{\pi}{3 \log 2}
\end{equation}
and the largest value of $\left\{\frac{2^{r}-2}{r(2^{r}-1)}\right\}_{r\in\mathbb{N}}$ is $1/3$. 

If $\tau(N)$ is the number of ones in the binary representation of $N$, and $\mathcal{N}\subset\mathbb{N}$ is a subsequence such that $(\tau(N))_{N\in\mathcal{N}}$ is bounded, then it follows from \eqref{eq:expsecordleo} that $(H_{1}(\alpha_{N,1})-N^{2} I_{1}(\sigma_{1}))_{N\in\mathcal{N}}$ is bounded, and therefore
\[
\lim_{N\in\mathcal{N}}\frac{H_{1}(\alpha_{N,1})-N^{2} I_{1}(\sigma_{1})}{\log N}=0.
\]
So in order to obtain non-zero limit points of the sequence \eqref{eq:secordexplamone}, such as the limits \eqref{eq:lpleo}, one needs to take subsequences $\mathcal{N}$ such that $\lim_{N\in\mathcal{N}}\tau(N)=\infty$.    

\begin{figure}[h]
\includegraphics[width=\textwidth]{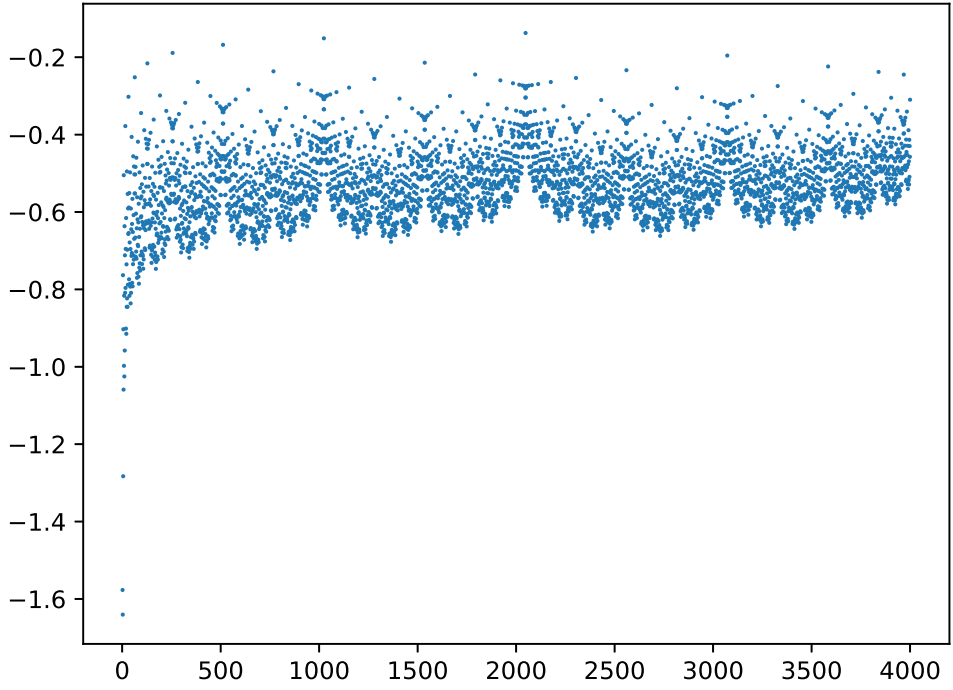}
\caption[]{{Plot of the sequence \eqref{eq:secordexplamone} for $2\leq N\leq 4000$, in the case $\lambda=1$.}}
\label{plotsecorder3}
\end{figure}

\section{The case $\lambda=2$}

\begin{theorem}\label{theo:caseletwo}
Let $d\geq 1$ be arbitrary. A sequence $(a_{n})_{n=0}^{\infty}\subset S^{d}$ is a greedy $\lambda$-energy sequence for $\lambda=2$, if and only if
\begin{equation}\label{eq:symmetry}
a_{2k+1}=-a_{2k},\quad\mbox{for every}\quad k\geq 0.
\end{equation}
If $(a_{n})_{n=0}^{\infty}\subset S^{d}$ is such a sequence, then for each $n\geq 1$,
\begin{equation}\label{eq:energy:lambda2}
\begin{aligned}
H_{2}(\alpha_{2n,2}) & =8 n^{2}\\
H_{2}(\alpha_{2n+1,2}) & =8 (n^2+n).
\end{aligned}
\end{equation}
For the potential \eqref{def:discrpot} we have, for each $n\geq 1$,
\[
U_{2n-1}(a_{2n-1})=U_{2n}(a_{2n})=4n.
\]
In particular,
\[
\lim_{N\rightarrow\infty}\frac{H_{2}(\alpha_{N,2})}{N^{2}}=\lim_{n\rightarrow\infty}\frac{U_{n}(a_{n})}{n}=2.
\]
\end{theorem}
\begin{proof}
The proof is based on the fact that for any $a, x\in S^{d}$,
\begin{equation}\label{eq:trivialid}
|x-a|^{2}+|x+a|^{2}=4.
\end{equation}

Let $(a_{n})_{n=0}^{\infty}\subset S^{d}$ be a greedy $\lambda$-energy sequence for $\lambda=2$. Clearly, \eqref{eq:symmetry} holds for $k=0$. Assume that \eqref{eq:symmetry} is valid for all $0\leq k\leq n$. By definition, the point $a_{2n+3}$ must maximize the function
\begin{align*}
\sum_{k=0}^{n}(|x-a_{2k}|^{2}+|x-a_{2k+1}|^{2})+|x-a_{2n+2}|^2 & =\sum_{k=0}^{n}(|x-a_{2k}|^{2}+|x+a_{2k}|^{2})+|x-a_{2n+2}|^2\\
& =4(n+1)+|x-a_{2n+2}|^2,\quad x\in S^{d}.
\end{align*}
Hence $a_{2n+3}=-a_{2n+2}$, which proves \eqref{eq:symmetry} for $k=n+1$. Conversely, it is clear that any sequence $(a_{n})_{n=0}^{\infty}\subset S^{d}$ that satisfies \eqref{eq:symmetry}, is a greedy $\lambda$-energy sequence for $\lambda=2$.

One proves \eqref{eq:energy:lambda2} by induction. This identity is valid for $n=1$. Assume that it is valid for $n$. Since
\[
H_{2}(\alpha_{2n+2})=H_{2}(\alpha_{2n+1})+2\sum_{j=0}^{2n}|a_{2n+1}-a_{j}|^{2}
\]
pairing $\{a_{2k},a_{2k+1}\}=\{a_{2k},-a_{2k}\}$ in the summation expression and using \eqref{eq:trivialid}, we obtain $H_{2}(\alpha_{2n+2})=8(n+1)^{2}$. Similarly, one checks that $H_{2}(\alpha_{2n+3})=8(n+2)(n+1)$.

The rest of the claims follow immediately.
\end{proof}

\begin{remark}
From \eqref{eq:symmetry} and \eqref{eq:trivialid} we deduce that the function $U_{2k}(x)\equiv 4k$ is constant on $S^{d}$, for every $k\geq 1$. 
\end{remark}

\begin{remark}
Our reviewer has indicated to us the interesting formula 
\begin{equation}\label{eq:sugrev1}
U_{n}(x)=2n-2\,\langle x,\sum_{k=0}^{n-1}a_{k}\rangle, \qquad x\in S^{d}, 
\end{equation}
valid for $\lambda=2$. This formula can be used for an alternative proof of Theorem~\ref{theo:caseletwo}.
\end{remark}

\begin{remark}\label{rmk:distrletwo}
If $(a_{n})_{n=0}^{\infty}\subset S^{d}$ is a greedy $\lambda$-energy sequence for $\lambda=2$, then the associated sequence $(\sigma_{N})$ defined in \eqref{def:countmeas} is not necessarily convergent. The following sequence on $S^{2}\subset\mathbb{R}^{3}$ illustrates this claim. Let
\[
y_{0}:=(1,0,0),\quad y_{1}:=(-1,0,0),\quad y_{2}:=(0,1,0),\quad y_{3}:=(0,-1,0),
\]
and let the first four points of the sequence be defined as
\[
a_{i}:=y_{i},\qquad 0\leq i\leq 3.
\]
The rest of the sequence is constructed inductively. Suppose that for an integer $m\geq 2$, the configuration $\alpha_{2^{m}}$ formed by the first $2^{m}$ points of the sequence, consists of $2^{m-2}$ points at each $y_{i}$, $0\leq i\leq 3$. Then, we have 
\begin{equation}\label{eq:meas:1}
\sigma_{2^{m}}=\frac{2^{m-2}}{2^{m}}(\delta_{y_{0}}+\delta_{y_1}+\delta_{y_2}+\delta_{y_3})=\frac{1}{4}(\delta_{y_{0}}+\delta_{y_1}+\delta_{y_2}+\delta_{y_3}).
\end{equation} 
We take the next $2^{m-1}$ points alternating between $y_0$ and $y_1=-y_{0}$, i.e., we define
\[
a_{2^{m}+k-1}:=(-1)^{k-1} y_{0},\qquad 1\leq k\leq 2^{m-1}.
\]
Then, the configuration $\alpha_{2^{m}+2^{m-1}}=\alpha_{3\cdot 2^{m-1}}$ consists of $2^{m-1}$ points at each $y_{0}$, $y_{1}$, and $2^{m-2}$ points at each $y_{2}$, $y_{3}$. Thus,
\begin{equation}\label{eq:meas:2}
\sigma_{3\cdot 2^{m-1}}=\frac{1}{3}(\delta_{0}+\delta_{1})+\frac{1}{6}(\delta_{y_{2}}+\delta_{y_{3}}).
\end{equation}
Now we take the next $2^{m-1}$ points alternating between $y_{2}$ and $y_{3}=-y_{2}$, i.e., we define
\[
a_{3\cdot 2^{m-1}+k-1}:=(-1)^{k-1} y_{2},\qquad 1\leq k\leq 2^{m-1}.
\]
With this definition, we have constructed the configuration $\alpha_{2^{m+1}}$, consisting of $2^{m-1}$ points at each $y_{i}$, $0\leq i\leq 3$. Therefore,
\[
\sigma_{2^{m+1}}=\sigma_{2^{m}},
\]
and the inductive construction is concluded. Since \eqref{eq:meas:1} and \eqref{eq:meas:2} are valid for every $m\geq 2$, we see that the sequence $(\sigma_{N})$ is not convergent. 

However, \eqref{eq:symmetry} implies that any convergent subsequence of $(\sigma_{N})$ will converge to a probability measure $\sigma$ on $S^{d}$ with center of mass at the origin, i.e., $\sigma$ satisfies
\begin{equation}\label{eq:centermass}
\int_{S^{d}}x_{i}\,d\sigma(x)=0,\qquad 1\leq i\leq d+1,\quad x=(x_{1},\ldots,x_{d+1}).
\end{equation}
\end{remark}

\section{The case $\lambda>2$}

\begin{lemma}\label{lem:caselg2}
Let $d\geq 1$, and let $a\in S^{d}$ be fixed. If $\lambda>2$, the function
\[
|x-a|^{\lambda}+|x+a|^{\lambda},\qquad x\in S^{d},
\]
attains its maximum value only at the points $x=a$ and $x=-a$.
\end{lemma}
\begin{proof}
In the case $d=1$, the claim is equivalent to show that the function
\[
f(\theta):=|e^{i\theta}-1|^{\lambda}+|e^{i\theta}+1|^{\lambda},\qquad 0\leq \theta\leq \pi,
\]
attains its maximum value on the interval $[0,\pi]$ only at the points $\theta=0$ and $\theta=\pi$. Indeed,
\begin{align*}
f(\theta) & =2^{\lambda}\left(\sin^{\lambda}\left(\frac{\theta}{2}\right)+\cos^{\lambda}\left(\frac{\theta}{2}\right)\right)\\
f'(\theta) & =\lambda\,2^{\lambda-1}\left(\sin^{\lambda-1}\left(\frac{\theta}{2}\right)\cos\left(\frac{\theta}{2}\right)-\cos^{\lambda-1}\left(\frac{\theta}{2}\right)\sin\left(\frac{\theta}{2}\right)\right),\qquad 0<\theta<\pi.
\end{align*}
Hence, the only critical point of $f$ on $(0,\pi)$ is $\theta=\pi/2$. We have $f(0)=f(\pi)=2^{\lambda}>f(\pi/2)=2^{1+\frac{\lambda}{2}}$, which proves the claim in the case $d=1$.

The proof can be completed now using induction on $d$. Assume the result holds on $S^{d-1}$, for some $d\geq 2$. Let $a\in S^{d}$ be fixed, and let $x\in S^{d}\setminus\{a,-a\}$. Since $d\geq 2$, there exists a hyperplane $\pi$ of $\mathbb{R}^{d+1}$ that passes through the origin and contains the points $a$, $-a$, $x$. The set $S^{d}\cap \pi$, which contains these three points, is isometric to $S^{d-1}$, which can be identified, for example, with the subset  of $S^{d}$ given by $S^{d}\cap[x_{d+1}=0]$. By the induction hypothesis, $|x-a|^{\lambda}+|x+a|^{\lambda}<2^{\lambda}$. Since $x\in S^{d}\setminus\{a,-a\}$ was taken arbitrarily, the result is proven.
\end{proof}

\begin{theorem}\label{theo:caselamg2}
Let $d\geq 1$ be arbitrary, and let $(a_{n})_{n=0}^{\infty}\subset S^{d}$ be a greedy $\lambda$-energy sequence for $\lambda>2$. Then, for each $k\geq 0$,
\begin{equation}\label{eq:struct}
\{a_{2k},a_{2k+1}\}=\{a_{0},-a_{0}\}.
\end{equation}
For each $n\geq 1$,
\[
\begin{aligned}
H_{\lambda}(\alpha_{2n,\lambda}) & =2^{\lambda+1}\,n^{2}\\
H_{\lambda}(\alpha_{2n+1,\lambda}) & =2^{\lambda+1}\,(n^2+n).
\end{aligned}
\]
For the potential \eqref{def:discrpot} we have, for each $n\geq 1$,
\[
U_{2n-1}(a_{2n-1})=U_{2n}(a_{2n})=n\,2^{\lambda}.
\]
In particular,
\[
\lim_{N\rightarrow\infty}\frac{H_{\lambda}(\alpha_{N,\lambda})}{N^{2}}=\lim_{n\rightarrow\infty}\frac{U_{n}(a_{n})}{n}=2^{\lambda-1},
\]
and  
\[
\sigma_{N,\lambda}\stackrel{\ast}{\longrightarrow}\frac{1}{2}\delta_{a_{0}}+\frac{1}{2}\delta_{-a_{0}}.
\]
\end{theorem}
\begin{proof}
Let $(a_{n})_{n=0}^{\infty}\subset S^{d}$ be a greedy $\lambda$-energy sequence for $\lambda>2$. Obviously, $a_{1}=-a_{0}$, so \eqref{eq:struct} holds for $k=0$. Assume that \eqref{eq:struct} is valid for all $0\leq k\leq n-1$, for some $n\geq 1$. Then $\alpha_{2n,\lambda}$ consists of $n$ points at $x=a_{0}$ and $n$ points at $x=-a_{0}$. So the point $a_{2n}$ must maximize the function
\[
\sum_{k=0}^{2n-1}|x-a_{k}|^{\lambda}=n(|x-a_{0}|^{\lambda}+|x+a_{0}|^{\lambda}),\qquad x\in S^{d}.
\]
By Lemma~\ref{lem:caselg2}, we have $a_{2n}=a_{0}$ or $a_{2n}=-a_{0}$. If $a_{2n}=a_{0}$, then $a_{2n+1}$ must maximize the function
\[
n(|x-a_{0}|^{\lambda}+|x+a_{0}|^{\lambda})+|x-a_{0}|^{\lambda},
\]
which clearly has a unique maximum at $x=-a_{0}$, so $a_{2n+1}=-a_{0}$. Similarly, $a_{2n}=-a_{0}$ implies that $a_{2n+1}=a_{0}$. So \eqref{eq:struct} is proved by induction.

The rest of the claims follow immediately. 
\end{proof}

\begin{remark}
For $\lambda>2$, the energy of any maximal distribution is $2^{\lambda-1}$. It follows from Theorem~\ref{theo:caselamg2} that for $\lambda>2$ and $N\geq 2$,
\[
H_{\lambda}(\alpha_{N,\lambda})-N^2\,2^{\lambda-1}=\begin{cases}
0 & \mbox{if}\,\,N\,\,\mbox{is even}\\
-2^{\lambda-1} & \mbox{if}\,\,N\,\,\mbox{is odd}.
\end{cases}
\]
For $n\geq 1$,
\[
U_{n}(a_{n})-n\,2^{\lambda-1}=\begin{cases}
0 & \mbox{if}\,\,n\,\,\mbox{is even}\\
2^{\lambda-1} & \mbox{if}\,\,n\,\,\mbox{is odd}.
\end{cases}
\]
These identities are also valid for $\lambda=2$.
\end{remark}

\bigskip

\noindent\textbf{Acknowledgements:} We thank the reviewer of this paper for sharing with us an interesting approach to the study of some of the problems discussed in this work.

\end{document}